\documentclass[11pt]{amsart}
\usepackage{amsmath,amscd,amssymb}
\usepackage[mathscr]{eucal}
\usepackage{amsmath}
\usepackage{enumerate}
\usepackage{color}

 \newtheorem{thm}{Theorem}[section]
 \newtheorem{cor}[thm]{Corollary}
 \newtheorem{lemma}[thm]{Lemma}
 \newtheorem{prop}[thm]{Proposition}

 \theoremstyle{definition}

 \theoremstyle{remark}
 \newtheorem{rem}[thm]{Remark}
 \newtheorem{rems}[thm]{Remarks}
 
\newtheorem{example}[thm]{Example}

\newtheorem{remark}[thm]{Remark}

\numberwithin{equation}{section}

\def\R{\mathbb{R}}
\def\C{\mathbb{C}}
\def\N{\mathbb{N}}

\def\a{\alpha}
\def\b{\beta}
\def\ep{\varepsilon}
\def\l{\lambda}

\def\Re#1{\operatorname{Re}#1}
\def\Im#1{\operatorname{Im}#1}

\def\Ker{\operatorname{Ker}}

\def\Bes{\mathcal{B}}
\def\Bov{\mathcal{E}}
\def\Bq{{\Bes_0}}
\def\fti{\mathcal{F}^{-1}}
\def\LT{\mathcal{LM}}
\def\HP{\mathrm{HP}}
\def\lt{\mathcal{L}}
\def\ssp{\mathcal{G}}
\def\RR{\mathrm{R}}

\def\Kop{K^\vartriangle}
\def\Qop{Q^\vartriangle}

\begin{document}

\title[Besov functional calculus]{The theory of Besov functional calculus: developments and applications to semigroups}

\author{Charles Batty}
\address{St. John's College\\
University of Oxford\\
Oxford OX1 3JP, UK
}

\email{charles.batty@sjc.ox.ac.uk}

%----------Author 2

\author{Alexander Gomilko}
\address{Faculty of Mathematics and Computer Science\\
Nicolas Copernicus University\\
Chopin Street 12/18\\
87-100 Toru\'n, Poland
}

\email{alex@gomilko.com}

  %%optional:  \curaddr{current address}%%
   %%optional:  \urladdr{website address}%

%----------Author 3
\author{Yuri Tomilov}
\address{
Institute of Mathematics\\
Polish Academy of Sciences\\
\'Sniadeckich 8\\
00-956 Warsaw, Poland
}

\email{ytomilov@impan.pl}

\begin{abstract}
We extend and deepen the theory of functional calculus
for semigroup generators, based on the algebra $\Bes$ of analytic Besov functions, which we initiated in a previous paper. 
In particular, we show that our construction of the calculus
is optimal in several natural senses. Moreover, we clarify the structure of $\mathcal B$
and identify several important subspaces in practical terms.
This leads to new spectral mapping theorems for operator semigroups
and to wide generalisations of a number of basic results from semigroup theory.
\end{abstract}



\thanks{This work was partially supported financially by a Leverhulme Trust Visiting Research Professorship and an NCN grant UMO-2017/27/B/ST1/00078, and inspirationally by the ambience of the Lamb \& Flag, Oxford.}

\date\today

\maketitle

\section{Introduction and preliminaries}

The theory of functional calculi is an indispensable building block of operator theory.  Its application to unbounded operators, in particular generators of $C_0$-semigroups, paved its way into various areas of analysis ranging from the theory of partial differential equations to ergodic theory.

Developing work by Peller \cite{Pel}, White \cite{Whi}, Vitse \cite{V1}, Haase \cite{Haase} and many others,
we constructed in \cite{BGT} a new functional calculus for the generators of bounded semigroups
on Hilbert and Banach spaces. The calculus is based on a Banach algebra $\mathcal B$, sometimes known as the analytic Besov algebra, which is canonically isomorphic to the classical holomorphic Besov space $B^0_{\infty,1}(\mathbb C_+)$.   The functional caclulus was called the $\mathcal B$-calculus in \cite{BGT}, and it allowed us to find a unified framework
for a number of operator-norm estimates in the literature, and to obtain several completely new ones.
A discussion of the $\mathcal B$-calculus from a historical perspective can be found in \cite{BGT}.
Here we just recall the very basic steps of its construction and clarify the ideas around them.
For a thorough and detailed treatment of various issues around the $\mathcal B$-calculus one may consult \cite{BGT}.

Let $A$ be a closed operator on a Banach space $X$, with dense domain $D(A)$.  
We assume that the spectrum $\sigma(A)$ is contained in $\overline\C_+:=\{z\in\C : \Re z \ge 0\}$ and that
\begin{equation} \label{8.1}
\sup_{\alpha >0} \alpha \int_{\mathbb R} |\langle (\alpha +i\beta + A)^{-2}x, x^* \rangle| \, d\beta <\infty
\end{equation}
for all $x \in X$ and $x^* \in X^*$.   By the Closed Graph Theorem, there is a constant $c$ such that
\begin{equation} \label{8.2}
 \frac{2}{\pi}\alpha \int_{\mathbb R} |\langle (\alpha +i\beta + A)^{-2}x, x^* \rangle| \, d\beta \le c \|x\|\,\|x^*\|
\end{equation}
for all $\a>0$, $x \in X$ and $x^* \in X^*$.    
We let $\gamma_A$ be the smallest value of $c$ such that (\ref{8.2}) holds.
At first glance, \eqref{8.2} might look artificial, but it is easy to show that
it is satisfied if either $-A$ generates a bounded $C_0$-semigroup and $X$ is a Hilbert space or 
$-A$ generates a bounded holomorphic semigroup on a Banach space $X$.   Moreover, as we showed in \cite{BGT}, \eqref{8.2} leads in a natural way to the choice of $\mathcal B$ as a function algebra
operating on those classes of semigroup generators.

Indeed, \eqref{8.2} says that the weak resolvents $g_{x,x^*}: z \mapsto \langle (z+A)^{-1}x, x^* \rangle$
of $A$ all belong to the Banach space $\mathcal E$ given by
\begin{equation} \label{defE}
\mathcal E:= \left\{g \in \operatorname{Hol}(\mathbb C_+): \|g\|_{\Bov_0} :=\sup_{\a >0} \a\int_{\mathbb R} |g'(\a+i\b)|\, d\b < \infty \right\}
\end{equation}
with the norm $\|\cdot\|_{\mathcal E}:=\|g\|_{\Bov_0} + |\lim_{\Re z \to \infty} g(z)|$.
(See \cite[Section 2.5]{BGT} for comments on $\mathcal E$, in particular concerning the existence of the limit).

Hoping to define $f(A)$ in a correct way and having in mind the Cauchy integral formula (and the classical Riesz-Dunford calculus), 
it is natural to try to pair
$g_{x,x^*}$ with functions from the dual space of $\mathcal E.$
This appeared to be a proper approach, but there was a number of difficulties to overcome while realising it.
Observing that by Green's formula one has
\begin{equation} \label{Green0}
\int_{0}^{\infty} \a \int_\R g'(\a-i\b) f'(\a+i\b)\, d\b\,d\a = \frac{1}{4} \int_\R  g(-iy)f(iy) \, dy
\end{equation}
for $f$ and $g$, holomorphic in $\overline {\mathbb C}_+$ and decaying sufficiently fast, together with their first derivatives in $\overline{\mathbb C}_+$,
we introduced (formally) a partial duality $\langle\cdot, \cdot \rangle_{\Bes}$ given by
\begin{equation}\label{dual}
\langle g, f \rangle_{\Bes}: = \int_{0}^{\infty} \a \int_\R g'(\a-i\b) f'(\a+i\b)\, d\b\,d\a.
\end{equation}
As  in \cite[Section 2.2]{BGT}, let $\mathcal B$ be the algebra  defined
by
\begin{equation}  \label{bdef0}
\mathcal B:= \left\{f \in \operatorname{Hol}\, (\mathbb C_+): \|f\|_\Bq:=\int_{0}^{\infty} \sup_{\b \in \R }|f'(\a+i\b)| \, d\a <\infty \right \}.
\end{equation}
One can easily show that every $f \in \mathcal B$ belongs to $H^\infty(\mathbb C_+),$ and $\mathcal B$ is a Banach algebra with the norm 
\begin{equation}\label{b-norm}
\|f\|_{\Bes}:=\|f\|_\infty+ \|f\|_\Bq.
\end{equation} 
Then  $\mathcal E$ can be paired  with $\mathcal B$ via $\langle g, f \rangle_{\Bes}$, in the sense that $\langle g, f \rangle_{\Bes}$ induces $\mathcal B \hookrightarrow \mathcal E^*$ (and $\mathcal E \hookrightarrow \mathcal B^*$) contractively.   However, the spaces $\mathcal B$ and $\mathcal E$  are neither the dual nor the predual of each other with respect to $\langle g, f \rangle_{\Bes}$.    It is instructive to remark here that the choices of $\mathcal B$ and of $\langle g, f \rangle_{\Bes}$ were equally important in \cite{BGT}, and a different duality would be of no use for our purposes.

Setting $g = r_z$ where $r_z(\l) = (z+\l)^{-1}$ for $z \in \overline{\C}_+$ and choosing good enough $f \in \mathcal B$
in \eqref{Green0}, 
one observes that the duality \eqref{dual} recovers the classical Cauchy integral formula:
\begin{equation}\label{Green}
\frac{2}{\pi} \langle r_z,f \rangle_\Bes = \frac{1}{2\pi} \int_\R \frac{f(iy)}{z-iy} \, dy = \frac{1}{2\pi i} \int_{i\R}\frac{f(\l)}{\l-z} \, d\l = f(z), \quad z \in \C_+.
\end{equation}
Cauchy's formula on the right-hand side of \eqref{Green} imposes stringent assumptions on $f$ and cannot hold for all $f \in \mathcal B$.  However, the left-hand side of \eqref{Green} requires less from $f$, and it suggests the following
reproducing formula for $\mathcal B$-functions $f$ proved in \cite[Proposition 2.20]{BGT}:
\begin{align}\label{bhp}
f(z)&= f(\infty) + \frac{2}{\pi} \langle r_z,f \rangle_\Bes \\
 &= f(\infty) - \frac{2}{\pi}  \int_0^\infty \int_{\R} \frac{\alpha f'(\alpha +i\beta)}{(z+\alpha -i\beta)^{2}} \, d\beta\,d\alpha, \qquad  z \in \overline{\C}_+. \notag 
\end{align}
The formula was of major importance in \cite{BGT}, and it will be crucial for us in this paper as well.

The algebra $\mathcal B$ is quite large. It properly contains the Hille-Phillips algebra $\mathcal {LM}$ of Laplace transforms of bounded measures on $[0,\infty)$.   The Hille-Phillips calculus is given by
\[
(\lt\mu)(A)x = \int_0^\infty T(t)x \, d\mu(t), \qquad x \in X,
\]
where $(T(t))_{t\ge0}$ is the bounded $C_0$-semigroup generated by $-A$.

To study $\mathcal B$ efficiently it is convenient to resort to its (non-closed) subalgebra
\begin{align}  \label{ssp}
\ssp &:= \left\{f \in H^\infty(\C_+) : \text{$\operatorname{supp} (\fti f^b)$ is a compact subset of $(0,\infty)$}\right \},
\end{align}
where $f^b$ denotes the boundary value of $f$ and $\mathcal F$ stands for the (distributional) Fourier transform.
Note that $\ssp$ consists of entire functions of exponential type, and moreover its norm-closure in $\Bes$ is 
\[\overline{\ssp} = \left\{ f \in \Bes:  f(\infty)=0 \right\}=:\mathcal B_0.\]
Thus $\ssp$ plays a similar role to that of the polynomials in the study of spaces of holomorphic functions on the unit disc.
Given a mere definition \eqref{bdef0} of $\mathcal B$, the density of $\ssp$ in $\mathcal B_0$ is a rather non-trivial fact, and the proof given in \cite[Proposition 2.10]{BGT} relies on fine properties of (holomorphic) semigroups and (isometric) groups of shifts associated to $\mathcal B$.

The formula \eqref{bhp} suggests the following definition of the functional calculus for $A$ based on the algebra $\mathcal B$.   For $f \in \Bes$ and $A$ as above,  set 
\begin{align}  \label{fcdef}
\lefteqn{\langle f(A)x, x^* \rangle} \quad\\
&=  f(\infty) \langle x, x^* \rangle - \frac{2}{\pi}  \int_0^\infty  \alpha \int_{\R}   \langle (\alpha -i\beta +A)^{-2}x, x^* \rangle  {f'(\alpha +i\beta)} \, d\beta\,d\alpha  \nonumber \\
&=  f(\infty) \langle x, x^* \rangle + \frac{2}{\pi} \langle g_{x,x^*}, f \rangle_\Bes,  \nonumber
\end{align}
for all $x \in X$ and $x^* \in X^*$. Then one infers that  $f(A)$ is a bounded linear mapping from $X$ to $X^{**}$, and that the linear mapping
\[
\Phi_A : \Bes \to \mathcal L(X,X^{**}), \qquad f \mapsto f(A),
\]
is bounded.
Such a definition leads to the following fundamental theorem  on the existence and basic properties of the $\mathcal B$-calculus.  The theorem was one of the main results in \cite{BGT}.

\begin{thm}\label{homomor_intro}
Let $-A$ be the generator of either a bounded $C_0$-semigroup on a Hilbert space $X$
or of a (sectorially) bounded holomorphic $C_0$-semigroup on a Banach space $X$.  Then the following hold.
\begin{enumerate}[\rm a)]
\item
The formula \eqref{fcdef} defines a bounded algebra homomorphism
\[
\Phi_A: \Bes \to  L (X),\qquad \Phi_A (f):=f(A),
\]
and
\begin{equation}\label{b_estim}
\|f(A)\|\le \gamma_A \|f\|_{\mathcal B}.
\end{equation}

\item The $\Bes$-calculus defined in {\rm  a)} (strictly) extends the Hille-Phillips (HP-) calculus, and it is compatible with
the holomorphic functional calculi for sectorial and half-plane type operators.

\item The spectral inclusion (spectral mapping, in the case of bounded holomorphic semigroups) theorem 
and a Convergence Lemma hold for $\Phi_A$.
\end{enumerate}
\end{thm}
In the case of holomorphic semigroups the estimate \eqref{b_estim} can be given a different, more explicit form (see \cite[Corollary 4.8]{BGT}).
The improvement is accomplished via a trick employing the integrability of $(\cdot+A)^{-2}$ along vertical lines in the uniform operator topology.

We underline that in contrast to other calculi considered in the literature, the $\mathcal B$-calculus is a Banach algebra
homomorphism. Thus it is the calculus in a sense adopted in classical operator theory, and this is a rare case when one deals with unbounded
operators.

The algebra $\mathcal B$ can also be defined in a standard way via an appropriate Littlewood-Paley decomposition
of its boundary values on $i\mathbb R$, see  \cite[Appendix]{BGT} for more details.
The main advantage of \eqref{b_estim} with $\|\cdot\|_{\mathcal B}$ defined by \eqref{b-norm} is that as a rule we are given 
a function $f \in \mathcal B$
 and not a Littlewood-Paley decomposition of $f$.  So it is very convenient to estimate the norm $\|f\|_{\mathcal B}$ and thus $\|f(A)\|$ in terms of the comparatively simple bound
for  $f'$ in the right half-plane. A number of concrete examples in \cite[Section 5]{BGT}, including functions such as $(z-1)^n(z+1)^{-n}$ (where $n \in \mathbb N$), $z^2(z+1)^{-2}e^{-1/z}$, $(\log (z+2))^{-2}$, and also those in the present paper,
justify this point. At the same time, the estimates for $\|f(A)\|$ by means of Littlewood-Paley decompositions of $f$ are rather involved, and they often rely on the Fourier multiplier theory even for comparatively simple $f$.

The route from \eqref{bhp} to Theorem \ref{homomor_intro} is not straightforward, and a formal insertion of $A$ instead of $z$ in \eqref{bhp} requires several steps of independent interest. Roughly, one extends $\Phi_A$ from $\mathcal {LM}$
where it coincides with the Hille-Phillips calculus, to the whole of $\mathcal B$  using the weak topology
induced on $\mathcal B$ by \eqref{dual} and applying several approximate unit arguments.
An essential point on the way is that $\mathcal {LM}$ appears to be dense in $\mathcal B$ in the $\mathcal E$-weak topology
(see \cite[Lemmas 2.13 and 2.19]{BGT}).

Once a general structure of the $\mathcal B$-calculus is established and its value has become clear,
it is natural to clarify its particular aspects and to reveal the relations between them.   This is the aim of the present paper.
The first natural question we address here is whether Theorem \ref{homomor_intro} is a really optimal result (for the corresponding classes of semigroup generators). In Section \ref{necuni}, we complement Theorem \ref{homomor_intro} by proving
a) the necessity of the resolvent condition \eqref{8.1} for the existence of $\Phi_A$;
b) the uniqueness of $\Phi_A$; and 
c) the (pointwise) multiplier algebra of $\mathcal B_0$ is $\Bes$, so $\Phi_A$ cannot be extended beyond $\Bes$ by multiplier methods.
 
As we mentioned above, the reproducing formula \eqref{bhp} was a crucial matter in \cite{BGT}, and it was obtained in \cite{BGT} by means of the $\mathcal E$-weak topology and an involved approximation procedure. 
However, in the literature, one gets similar formulas 
 by extending a ``standard'' duality
on $i\mathbb R$ into the right half-plane by Green's formula, as in \eqref{Green}. We establish a new version of Green's formula  for  $f \in H^1(\mathbb C_+)$ and $g \in \mathcal B_0$, and we show in Theorem \ref{rep} that the formula implies   \eqref{dual}.
Moreover, we obtain several other reproducing formulas. This clarifies \eqref{dual} in our context, and 
enlightens the construction of the $\mathcal B$-calculus.

So far, the only practical way to construct $\mathcal B$-functions was to resort to $\mathcal{LM} \subset \mathcal B$.
Thus, the advantages of a larger algebra $\mathcal B$ were hardly feasible.   In Proposition \ref{L1} and Remark \ref{L1L}, we give a recipe for constructing $\mathcal B$-functions and show, in particular, that any pair of functions $g_1 \in L^1(\mathbb R_+)$ and $g_2 \in L^\infty(\mathbb R)$ gives rise to a function $f \in \Bes$ given by
\[
f(z) = -\frac{2}{\pi} \int_0^\infty  \int_\R \frac{\a g_1(\a) g_2(\b)}{(z+\a-i\b)^2} \,d\b\,d\a, \qquad z \in \C_+.
\]
This approach also yields several useful statements on approximation of $\mathcal B$-functions by means of their truncations, as in Propositions \ref{Pr26} and \ref{Pr28}.  In turn, those results allow us to identify in Theorem \ref{CCCA} the closures of $\mathcal {LM}$ and of its important subspaces in the $\mathcal B$-norm.   By combining our function-theoretic considerations with the $\mathcal B$-calculus, we prove a new spectral mapping theorem (Theorem \ref{rajchman}) for the  HP (and then $\mathcal B$)-calculus, involving singular measures.  Such a spectral mapping theorem is the first result of this kind in the literature (to our knowledge).   It is instructive to recall that the well-known failure of the spectral mapping theorem for $C_0$-semigroups is a failure within the HP-calculus for delta-measures.

Since the extension of the HP-calculus to $\mathcal B$ is not isometric, it is essential
to understand the gap between the two calculi in several situations of interest.   The gap can be substantial, of polynomial growth, for simple families of functions, see Section \ref{compnorm} for more details.  Hence the $\mathcal B$-calculus offers an essential improvement over the HP-calculus, at least as far as the norm-estimates are concerned. 

Finally in this paper, we extend several classical results from semigroup theory on Hilbert spaces into the setting of the $\mathcal B$-calculus.   With some exaggeration, the theory of $C_0$-semigroups can be considered as the study 
of properties of exponential functions within  appropriate functional calculi.
We make a major step towards justifying this point of view. Recall that the semigroup generator $-A$ is (initially) defined as the right-hand side derivative of the exponential function $[0,\infty) \ni t\mapsto e^{-tA}$ on its natural domain.
We prove in Theorem \ref{Gen2} that if $-A$ generates a bounded $C_0$-semigroup on a Hilbert space, then $-A$  can be identified in precisely the same way by replacing $(e^{-tA})_{t \ge 0}$  with any one-parameter family
$(f(tA))_{t \ge 0}$,  where $f \in \mathcal B$ is such that  $f' \in \Bes$ and $f'(0)=-1$.
Moreover, the famous norm-continuity and exponential stability resolvent criteria for Hilbert space semigroups
can be extended to a similar general form. Our general form of the norm-continuity criterion  given in Theorem \ref{Norm1}
seems to be especially revealing.   We also obtain a version for functions in $\Bes$ of the complex inversion formula for $C_0$-semigroups.  The functional calculus ideology in the study of semigroup properties is clearly seen in those results and also in the treatment of spectral mapping properties in Section \ref{SMT}. We expect it will be useful in other similar instances.

We note here that more advanced functional calculi for generators of holomorphic semigroups have been constructed very recently in \cite{BGT21}.

\subsection*{Notation}
Throughout the paper, we shall use the following notation:
\begin{enumerate}[\phantom{X}]
\item $\R_+ :=[0,\infty)$,
\item $\C_+ := \{z \in\C: \Re z>0\}$, $\overline{\C}_+ = \{z \in\C: \Re z\ge0\}$,
\item $\Sigma_\theta := \{z\in\C: z \ne 0, |\arg z|<\theta\}$ for $\theta \in (0,\pi)$, 
\item $\mathrm{R}_a := \{z\in\C: \Re z > a\}$.
\item
\end{enumerate}
For $f : \C_+ \to \C$ and $s \in \R$, we write
\begin{gather*}
f(\infty) = \lim_{\Re z \to \infty} f(z),  \quad \|f\|_\infty = \sup_{z \in \C_+} |f(z)|, \\
 f^b(s) =  \lim_{t\to0+} f(t+is),
\end{gather*}
whenever these  exist in $\C$.

\noindent
For $a \in \overline{\C}_+$, we define functions on $\C_+$ by
\[
e_a(z) = e^{-az}, \; r_a(z) = (z+a)^{-1}.
\]

\noindent
We use the following notation for spaces of functions or measures, and transforms, on $\R$ or $\R_+$:
\begin{enumerate}[\phantom{X}]
\item $\mathcal{S}(\R)$ denotes the Schwartz space on $\R$,
\item  $\operatorname{Hol}(\Omega)$ denotes the space of holomorphic functions on an open subset $\Omega$ of $\C_+$,
\item $H^\infty(\C_+)$ and $H^1(\C_+)$ are the standard Hardy spaces on the (right) half-plane.
\item $M(\R_+)$ denotes the Banach algebra of all bounded Borel measures on $\R_+$ under convolution.   We identify $L^1(\R_+)$ with a subalgebra of $M(\R_+)$ in the usual way.  We write $\lt\mu$ for the Laplace transform of $\mu \in M(\R_+)$.
\end{enumerate}

\noindent
For normed vector spaces $X$ and $Y$, $L(X,Y)$ denotes the space of all bounded linear operators from $X$ to $Y$, and $L(X) = L(X,X)$.   The term ``operator on $X$'' will be used to indicate a linear operator $A : D(A)\to X$ on $X$, where the domain $D(A)$ is a linear subspace of $X$.   In this paper $D(A)$ will normally be dense in $X$ and $A$ will be closed.   The spectrum and resolvent set of $A$ will be denoted by $\sigma(A)$ and $\rho(A)$, respectively.    Note that $A$ is automatically closed if $X$ is a Banach space and $\rho(A)$ is non-empty.

\subsection*{Properties of $\Bes$}

Here we recall from \cite[Section 2]{BGT} a few basic facts about the analytic Besov algebra $\Bes$ and its norms defined in \eqref{bdef0} and \eqref{b-norm}. 

Every function $f \in \Bes$ is bounded and uniformly continuous on $\C_+$, and 
\[
f(\infty) := \lim_{\Re z \to \infty} f(z)
\]
exists.  Moreover, $f$ extends to a uniformly continuous function (also denoted by $f$) on $\overline{\C}_+$, and the boundary function $f^b$ satisfies
\[
f^b(s) = f(is) :=  \lim_{z\in\C_+, z\to is} f(z), \qquad s\in\R.
\]

The pair $(\Bes, \|\cdot\|_\Bes)$ is a Banach algebra containing the constant functions.   We shall also consider 
\[
\Bq := \{f\in\Bes: f(\infty)=0\},
\]
which is a closed ideal of $\Bes$.

For $f\in\Bes$, 
\begin{equation} \label{eqnm}
\|f\|_\infty \le |f(\infty)| + \|f\|_\Bq.
\end{equation}
Hence  $|f(\infty)| + \|f\|_\Bq$ is an equivalent norm on $\Bes$, but it is not an algebra norm.

Since $\Bes \subset H^\infty(\C_+)$, we have standard  Cauchy formulas for $f\in \Bes$:
\[
f(z) = \frac{f(\infty)}{2} - \frac{1}{2\pi} \int_\R \frac{f^b(s)}{is-z} \,ds,   \quad  f'(z) = - \frac{1}{2\pi} \int_\R \frac{f^b(s)}{(is-z)^2} \,ds,
\]
where the first integral is a principal value integral, and the second integral provides the estimate
 \begin{equation} \label{f'est}
 |f'(z)| \le \frac{\|f\|_\infty}{2 \Re z}, \qquad z \in \C_+.
 \end{equation}
 Thus $f'(\infty)=0$ for all $f \in \Bes$.  We also have the Poisson reproducing formula:
\begin{equation} \label{poisson}
 f(x+iy) =  \frac{1}{\pi} \int_\R P(x,y-s) f^b(s) \, ds, \qquad P(x,y) = \frac{x}{\pi(x^2+y^2)}.
 \end{equation}

We recall from \cite[Lemmas 2.6 and 4.6]{BGT} some properties of shifts and rescalings on  $\Bes$.

\begin{lemma} \label{shifts01}
Let
\[
  (T_\Bes(a)f)(z):=f(z+a), \qquad  f \in \Bes, \; a\in \overline\C_+, \; z\in \C_+.
	\]
\begin{enumerate}[\rm1.]
\item   For each $f \in \Bes$,
\[
\|T_\Bes(a)f\|_\Bes \le \|f\|_\Bes, \quad \lim_{a\in\overline\C_+, a\to0}\,\|T_\Bes(a)f-f\|_\Bes=0.
\]
\item The family $(T_\Bes(a))_{a\in\C_+}$ is a holomorphic 
$C_0$-semigroup of contractions on $\Bes$.
\item  Let $-A_\Bes$ be the generator of the $C_0$-semigroup $(T_\Bes(t))_{t\ge0}$ on $\Bes$.  Then
\[
D(A_\Bes) = \{ f \in \Bes: f' \in \Bes\},  \quad A_\Bes f = -f'.
\]
\item The generator of the $C_0$-group $(T_\Bes(-is))_{s\in\R}$ is $iA_\Bes$.
\item $\sigma(A_\Bes) = \R_+$, and the range of $A_B$ is dense in $\Bq$.
\item The family $(tA_\Bes T_\Bes(t))_{t>0}$ is uniformly bounded in $L(\Bq)$, and it converges in the strong operator topology to $0$ as $t\to0+$ and as $t\to\infty$.
\item \label{012} Let $f\in\Bes$ and $(S_\Bes(b)f)(z) = f(bz)$, $b>0$.  Then $S_\Bes(b)f \in \Bes$ and $\|S_\Bes(b)f\|_\Bes = \|f\|_\Bes$.
\end{enumerate}
\end{lemma}

\begin{proof}  Most of these statements are proved in \cite[Lemmas 2.6 and 4.6]{BGT}.   The density of the range of $A_\Bes$ is 
shown in the proof of \cite[Proposition 2.10]{BGT}.  In the sixth part, the boundedness is a standard fact about bounded holomorphic semigroups, the convergence to $0$ as $t\to0+$ follows from the boundedness and the density of the domain of $A_\Bes$, and the convergence as $t\to\infty$ follows from the boundedness of $(t^2A_\Bes^2 T_\Bes(2t))_{t>0}$ and the density of the range of $A_\Bes$.
\end{proof}

\section{Reproducing formulas} \label{sectB}

In this section we give further results relating to the reproducing formula \eqref{bhp}, including an alternative proof  of \eqref{bhp} and some variants.  We begin by recalling from \cite[Section 2]{BGT} some basic facts about the analytic Besov algebra $\Bes$ and its norms defined in \eqref{bdef0} and \eqref{b-norm}. 

Let $\Bov$ be the space of holomorphic functions on $\C_+$ defined in \eqref{defE}, and let $\|\cdot\|_{\Bov_0}$ and $\|\cdot\|_\Bov$ be as defined there.   Note the partial duality $\langle g,f \rangle_\Bes$ defined in \eqref{dual}, and the reproducing formula defined in \eqref{bhp}, for $g \in \Bov$ and $f \in \Bes$.

 In \cite[Proposition 2.16]{BGT} we showed that $H^1(\C_+) \subset \Bov$.   Recall that if $g \in H^1(\mathbb C_+)$ then the boundary function $g^b$ exists a.e.\ on $\mathbb R$ and 
 \begin{equation} \label{H1sc}
 \lim_{a \to 0+}\|g^b-g(a+i\cdot)\|_{L^1(\mathbb R)}=0.
 \end{equation}
We will need the following simple properties of functions from $H^1(\C_+)$ and $\mathcal B$, firstly to provide a simpler form of \eqref{dual} in Proposition \ref{P1} and then to give a more direct proof of \eqref{bhp} in Theorem \ref{rep}.

\begin{prop}\label{P1}

\begin{enumerate}[\rm1.]

\item Let $g\in H^1(\mathbb C_{+})$, $\omega>0$. Then
\begin{equation}\label{A1}
\lim_{|z|\to\infty, z \in \RR_\omega}\,\bigl(|g(z)|+|g'(z)| \bigr)=0.
\end{equation}
\item Let $f\in \mathcal B$. Then
\begin{equation}\label{B1}
\lim_{x\to 0}\, x\sup_{y\in \mathbb R}\,|f'(x+iy)|=0,\qquad
\lim_{x\to\infty }\, x\sup_{y\in \mathbb R}\,|f'(x+iy)|=0.
\end{equation}
\end{enumerate}
\end{prop}

\begin{proof}
The property \eqref{A1}  follows from the integral representations
\[
g(z)=\frac{1}{2\pi}\int_\R \frac{g^b(s)}{z-is} \,ds,\qquad
g'(z)=-\frac{1}{2\pi}\int_\R \frac{g^b(s)}{(z-is)^2} \,ds,\quad z\in \mathbb C_{+},
\]
and the dominated convergence theorem.

To prove \eqref{B1}, it suffices to note that the function $x  \mapsto \sup_{y\in \mathbb R}\,|f'(x+iy)|$ is decreasing on $(0,\infty)$, by the maximum principle, so that
\[
x\sup_{y\in \mathbb R}\,|f'(x+iy)|\le 2\int_{x/2}^x
\sup_{y\in \mathbb R}\,|f'(t+iy)|\,dt\to 0
\]
as $x\to 0$ or $x\to\infty$.
\end{proof}

For $0<a<b<\infty$ and $0<c\le \infty$, let
\[
R(a,b;c):=\{(x,y):\,x\in [a,b],\,|y|\le c\}.
\]
If $F\in C^2(R[a,b;c])$ where $c$ is finite, then
by Green's formula,
\begin{align}  \label{green}
\lefteqn{\int_{R(a,b;c)} x \Delta F(x,y)\,dx \, dy} \\
&= \null -\int_{R(a,b;c)} \frac{\partial F}{\partial x}(x,y)\,dx \, dy
-a\int_{-c}^c \frac{\partial F}{\partial x}(a,y)\,dy + b\int_{-c}^c \frac{\partial F}{\partial x}(b,y)\,dy  \notag\\
&\null \hskip40pt
+\int_a^b x \frac{\partial F}{\partial y}(x,c)\,dx-\int_a^b x \frac{\partial F}{\partial y}(x,-c)\,dx \notag\\
&= \int_{-c}^c F(a,y)\,dy-\int_{-c}^c F(b,y)\,dy
-a\int_{-c}^c  \frac{\partial F}{\partial x}(a,y)\,dy \notag\\
&\null \hskip40pt + b\int_{-c}^c  \frac{\partial F}{\partial x}(b,y)\,dy
+\int_a^b x \frac{\partial F}{\partial y}(x,c)\,dx-\int_a^b x \frac{\partial F}{\partial y}(x,-c)\,dx. \notag
\end{align}

The following proposition partially generalises \cite[Lemma 17]{Taib} which
proved the result for harmonic functions  $g$ and $f$ on $\mathbb C_+$ defined by the Poisson integrals of $G \in L^p(\mathbb R)$ and $F \in L^{p'}(\mathbb R)$,  where $1<p<\infty, \,1/p +1/p'=1$.
It was remarked in \cite[p.456]{Taib} that the lemma can be extended to the case when $G \in L^1(\mathbb R)$ and either $F \in C_0(\mathbb R)$, or $F$ is bounded and uniformly continuous and $\{t\in\R: |F(t)|>\ep\}$ has finite Lebesgue measure for every $\ep >0$.
   Our statement goes beyond these assumptions, as some functions in $\mathcal B$, such as $e_a, \, a \in \overline{\C}_+$, satisfy neither of the conditions from \cite{Taib}.    Other, less general, versions of Taibleson's result can be found in \cite{FS}.

\begin{prop}\label{Dd}
Let $g\in H^1(\mathbb C_{+})$ and $f\in \mathcal B_0$.   Then
\begin{equation}\label{22}
\langle g,f\rangle_{\mathcal B}=\frac{1}{4}\int_\R g^b(-y)f^b(y)\,dy.
\end{equation}
\end{prop}

\begin{proof}
We apply Green's formula \eqref{green} to
\[
F(x,y):=g(x-iy)f(x+iy), \qquad x+iy \in \mathbb C_+.
\]
Note that
\begin{align*}
\Delta F(x,y)&=4g'(x-iy)f'(x+iy),
\end{align*}
and
\begin{equation}\label{Note}
\int_0^\infty x\int_\R |g'(x-iy)f'(x+iy)|\,dy \, dx<\infty
\end{equation}
since $g \in \mathcal E$ and $f \in \Bes$.  Using \eqref{green} and
\begin{multline*}
\left|\int_{-c}^c  g'(r-iy)f(r+iy)\,dy\right|  \\
\le |g(r-ic)f(r+ic)| + |g(r+ic)f(r-ic)| +\int_{-c}^c  |g(r-iy)f'(r+iy)|\,dy
\end{multline*}
for $r=a$ and $r=b$,
we obtain
\begin{align}
\lefteqn{\left|4\int_{R[a,b;c]} x g'(x-iy)f'(x+iy)\,dx dy-\int_{-c}^c g(a-iy)f(a+iy)\,dy
\right|} \label{C1}\\
&\le
\int_{-c}^c |g(b-iy)f(b+iy)|\,dy + 2a\int_{-c}^c  |g(a-iy)f'(a+iy)|\,dy \notag \\
&\null\hskip20pt  + 2b\int_{-c}^c  |g(b-iy)f'(b+iy)|\,dy \, +a|g(a-ic)f(a+ic)|\notag \\
&\null\hskip20pt +a|g(a+ic)f(a-ic)|+b|g(b-ic)f(b+ic)|+ b|g(b+ic)f(b-ic)|\notag \\
&\null\hskip20pt
+ \int_a^b x \left(|g'(x-ic)f(x+ic)|+|g(x-ic)f'(x+ic)|\right)\,dx \notag \\
&\null\hskip20pt + \int_a^b x \left(|g'(x+ic)f(x-ic)|+|g(x+ic)f'(x-ic)|]\right)\,dx. \notag
 \end{align}

Now let $c\to\infty$ in (\ref{C1}).  By our assumptions,
 \begin{equation}\label{gf}
 f\in H^\infty(\mathbb C_{+}), \; f' \in  H^{\infty}(\RR_a)\quad \text{and} \quad g \in H^1(\mathbb C_+),\; g'\in H^1(\RR_a).
 \end{equation}
From this and Proposition \ref{P1}(1), it follows that
\[
\lim_{c\to\infty}\,\int_a^b x \left(|g'(x\pm ic)f(x\mp ic)|+|g(x\pm ic)f'(x\mp ic)|\right)\,dx=0,
\]
and
\[
\lim_{c\to\infty}\,
|g(r-ic)f(r+ic)| = \lim_{c\to\infty}|g(r+ic)f(r-ic)|=0,\quad r=a,b.
\]
Using \eqref{gf} again,  and also (\ref{Note}), we conclude that
\begin{align}
\lefteqn{\left|4\int_{R[a,b;\infty]} x g'(x-iy)f'(x+iy)\,dx dy-\int_\R g(a-iy)f(a+iy)\,dy
\right|} \label{C3}\\
& \le
\int_\R |g(b-iy)f(b+iy)|\,dy
+2a\int_\R |g(a-iy)f'(a+iy)|\,dy\notag \\
&\null\hskip30pt + 2b\int_\R  |g(b-iy)f'(b+iy)|\,dy.\notag
\end{align}
Next, since $f(\infty)=0$ and $g\in H^1(\mathbb C_{+})$,
\[
\lim_{b\to\infty}\,\int_\R |g(b-iy)f(b+iy)|\,dy=0.
\]
Moreover, in view of Proposition \ref{P1}(2),
\[
\lim_{a\to 0+}\,a\int_\R |g(a-iy)f'(a+iy)|\,dy\le
\|g\|_{H^1}\lim_{a\to 0+}\,a\sup_{y\in \mathbb R}\,|f'(a+iy)|=0,
\]
and
\[
\lim_{b\to \infty}\,b\int_\R |g(b-iy)f'(b+iy)|\le
\|g\|_{H^1}\lim_{b\to \infty}\,b\sup_{y\in \mathbb R}\,|f'(b+iy)|=0.
\]
The assertion (\ref{22}) follows, on letting $a\to 0+$ and $b\to\infty$  in (\ref{C3}) and using Lemma \ref{shifts01}(1) and \eqref{H1sc}.  
\end{proof}

Using the duality formula \eqref{22}, we now give an alternative proof
of the reproducing formula \eqref{bhp} for $\mathcal B$ from \cite[Proposition 2.20]{BGT}.

\begin{thm}\label{rep}
Let $f\in \mathcal B$. Then the reproducing formula holds:
\begin{equation*} 
f(z)=f(\infty)-\frac{2}{\pi}
\int_0^\infty \alpha\int_\R \frac{f'(\alpha+i\beta)}{(z+\alpha-i\beta)^2}\,d\beta \,d\alpha,\qquad z\in \overline{\mathbb C}_{+}.
\end{equation*}
\end{thm}

\begin{proof}
We may assume that $f(\infty)=0$.  For $z \in \mathbb C_{+}$,  $r_z^2\in H^1(\mathbb C_{+})$ and by Proposition \ref{Dd} and Cauchy's theorem,
\[
-2\int_0^\infty \alpha\int_\R \frac{f'(\alpha+i\beta)}{(z+\alpha-i\beta)^3}\,d\beta \, d\alpha
=\frac{1}{4}\int_\R \frac{f^b(\beta)}{(z-i\beta)^2}\,d\b =-\frac{\pi }{2}f'(z),
\]
for every $z\in \mathbb C_{+}$.  Integrating both sides of the above equality with respect to $z$ along horizontal lines to $\infty$, we obtain
\[
-\int_0^\infty \alpha\int_\R \frac{f'(\alpha+i\beta)}{(z+\alpha-i\beta)^2}\,d\beta \, d\alpha=\frac{\pi}{2}f(z).  
\]
By the dominated convergence theorem, the reproducing formula gives a continuous function of $z$ on $\overline{\C}_+$, so the formula holds on $\overline{\C}_+$.
\end{proof}

The following representations are variants of \eqref{bhp}.

\begin{prop} \label{brim}
If $f \in \mathcal B_0$, then
\begin{align*}
f(z)&=-\frac{4}{\pi}\int_0^\infty \alpha\int_\R \frac{{\rm Re}\,f'(\alpha+i\beta)}{(z+\alpha-i\beta)^2}\,d\beta\,d\alpha\\
&=-\frac{4i}{\pi}\int_0^\infty \alpha\int_\R \frac{{\rm Im}\,f'(\alpha+i\beta)}{(z+\alpha-i\beta)^2}\,d\beta\,d\alpha,\qquad z\in \C_{+}.
\end{align*}
\end{prop}

\begin{proof}
By the reproducing formula \eqref{bhp} for $\Bes$,
\begin{equation}\label{BB2}
f(z) = - \frac{2}{\pi}\int_0^\infty \alpha\int_\R \frac{f'(\alpha+i\beta)}{(z+\alpha-i\beta)^2} \,d\b\,d\alpha,
\end{equation}
for every $z=x+iy\in \C_{+}$.  
Moreover, by Cauchy's theorem,
\[
\int_\R \frac{f'(\a+i\b)}{\overline{z} +\a+i\b}  \,d\b =0,  \quad \a>0.
\]
Hence
\begin{equation} \label{BBAc}
0 =- \frac{2}{\pi}\int_0^\infty \alpha\int_\R \frac{\overline{f'(\alpha+i\beta)}}{(z+\alpha-i\beta)^2} \,d\b\,d\alpha.
\end{equation}
Since
\[
f' + \overline{f}' = 2 \Re f', \qquad f' - \overline{f}' = 2i \Im f',
\]
the formulas follow.
\end{proof}

Moreover, we can derive a reproducing formula for functions from $\mathcal B$ using their second derivatives instead of the first derivatives.

\begin{prop}\label{R3}
Let $f\in \mathcal B_0$, and $z = x+iy \in \overline{\C}_+$. Then
\begin{equation}\label{LL1}
f(z)=\frac{4}{\pi}\int_0^\infty \alpha (x+\alpha) \int_\R
\frac{f''(\alpha+i\beta)}{(x+\alpha)^2+(y-\beta)^2}\,d\beta \,d\alpha.
\end{equation}
\end{prop}

\begin{proof}
Using \eqref{BB2} and \eqref{BBAc}, we see that
\begin{align*}
{f(z)}
&= -\frac{2}{\pi}\int_0^\infty \alpha\int_\R \frac{f'(\alpha+i\beta)\,d\beta}{(z+\alpha-i\beta)^2}\,d\alpha - \frac{2}{\pi}\int_0^\infty \alpha\int_\R \frac{f'(\alpha+i\beta)\,d\beta}{(\overline{z}+\alpha+i\beta)^2}\,d\alpha \\
&=-\frac{4}{\pi}\int_0^\infty \alpha\int_\R \left({\rm Re}\,\frac{1}{(z+\alpha-i\beta)^2}\right) f'(\alpha+i\beta)\,d\beta\,d\alpha.
\end{align*}
Now
\begin{align*}
\Re\frac{1}{(z+\alpha-i\beta)^2}&=\frac{(x+\alpha)^2-(y-\beta)^2}{((x+\alpha)^2+(y-\beta)^2)^2}
= -\pi \frac{\partial}{\partial \alpha} (P(x+\alpha,y-\beta)),
\end{align*}
where
$P$ is the Poisson kernel.
Hence,
\begin{align*}
f(z) &=
- 4\int_\R \int_0^\infty P(x+\alpha,y+\beta) \left(f'(\alpha+i\beta)+\alpha f''(\alpha+i\beta)]\right)\,d\alpha\,d\beta \\
&= -4 \int_0^\infty \int_\R
P(x+\alpha,y+\beta)f'(\alpha+i\beta)\,d\beta \,d\alpha  \notag\\\
&\null\hskip40pt -\frac{4}{\pi}\int_0^\infty \alpha (x+\alpha) \int_\R
\frac{f''(\alpha+i\beta)}{(x+\alpha)^2+(y-\beta)^2} \,d\beta\,d\alpha\notag \\
&= - 4\int_0^\infty f'(x+2\alpha+iy)\,d\alpha \nonumber \\
& \null\hskip40pt -\frac{4}{\pi}\int_0^\infty \alpha (x+\alpha) \int_\R
\frac{f''(\alpha+i\beta)}{(x+\alpha)^2+(y-\beta)^2}\,d\beta \,d\alpha\notag \\
& =2f(z)-\frac{4}{\pi}\int_0^\infty \alpha (x+\alpha) \int_\R
\frac{f''(\alpha+i\beta)}{(x+\alpha)^2+(y-\beta)^2} \,d\beta\,d\alpha\notag,
\end{align*}
using \eqref{poisson} in the penultimate equality.  This yields \eqref{LL1}.
\end{proof}

\section{Constructing and approximating Besov functions} \label{approx}

In this section we first show how all functions in $\Bes$ can be constructed from a class $\mathcal{W}$ of measurable functions on $\C_+$ satisfying a condition similar to \eqref{bdef0}.     Later in this section we use this procedure to approximate functions in $\Bes$ in various ways, by lifting cut-off operators on $\mathcal W$ to the space $\Bes$.  

For $\varphi \in L^p(\R), \, 1 \le p \le \infty$, and $\a>0$, we define
\begin{equation}\label{ksigma}
G_{\alpha,\varphi}(z):=\int_\R \frac{\varphi(\beta)}{(z+\alpha-i\beta)^2}\,d\beta,\qquad z\in \C_{+}.
\end{equation}

Let $\mathcal W$ be the space of all (equivalence classes of) measurable functions $g:\mathbb C_{+}\to \mathbb C$ such that
\begin{equation}\label{A}
\|g\|_{\mathcal W}:=\int_0^\infty \underset{\beta\in \mathbb R}{\mathrm{ess\,sup}}\,|g(\alpha+i\beta)|\,d\alpha<\infty,
\end{equation}
with the norm given by (\ref{A}).    Then $\mathcal W$ is a Banach space, but this fact is not needed in this paper.   Clearly,
\begin{equation}\label{wb0}
\mathcal{W} \cap \operatorname{Hol}(\C_+) = \{ f' : f \in \Bes\}, \qquad \|f'\|_{\mathcal W} = \|f\|_\Bq.
\end{equation}

 For $g \in \mathcal{W}$, let
 \begin{align*}
(Q g)(z):&= - \frac{2}{\pi}\int_0^\infty \alpha\int_\R \frac{g(\alpha+i\beta)}{(z+\alpha-i\beta)^2}\,d\beta\,d\alpha,\\
&= - \frac{2}{\pi} \int_0^\infty \alpha G_{\a, \varphi_{\a,g}}(z) \,d\a, \qquad z\in \mathbb C_{+},
\end{align*}
where $\varphi_{\a,g}(\b) = g(\a+i\b)$.   

If $f \in \Bes$, the reproducing formula for $f$ given in \cite[Proposition 2.20]{BGT}, \eqref{bhp} and Theorem \ref{rep} can now be written as
\begin{equation} \label{repq}
f = f(\infty) + Q(f').
\end{equation} 
The following result shows that $Q$ maps the whole of $\mathcal{W}$ into $\Bq$.

\begin{prop}\label{L1}
{\rm1.}  Let $\varphi \in L^1(\R)$, and $\a>0$.  Then  $G_{\a,\varphi} \in \Bq$ and
\begin{equation} \label{B09}
G_{\a,\varphi} = \int_\R \varphi(\b) r_{\a-i\b}^2 \, d\b,
\end{equation}
where the right-hand side exists as a $\Bes$-valued Bochner integral.

\noindent
{\rm2.}  Let $\varphi\in L^\infty(\R)$, and $\alpha>0$.   Then $G_{\alpha, \varphi} \in \Bq$ and
\begin{equation}\label{B11}
\|G_{\alpha,\varphi}\|_{\Bq}\le \frac{4}{\alpha}\|\varphi\|_{L^\infty}, \qquad \|G_{\alpha,\varphi}\|_{\Bes}\le \left(\frac{4+\pi}{\alpha}\right) \|\varphi\|_{L^\infty}.
\end{equation}

\noindent
{\rm3.}  The operator $ Q \in L (\mathcal W, \mathcal B_0)$ and
\begin{equation}\label{BC}
\|Qg \|_{{\mathcal B_0}}\le \frac{8}{\pi}\|g\|_{\mathcal{W}}, \qquad
\|Qg \|_{{\mathcal B}}\le 2\left(1+\frac{4}{\pi}\right) \|g\|_{\mathcal{W}}.
\end{equation}
Moreover $Q$ maps $\mathcal{W} \cap \operatorname{Hol}(\C_+)$ onto $\Bq$.
\end{prop}

\begin{proof}  
1.  First recall that $r_{\lambda} \in \Bq$ for each $\lambda \in \mathbb C_+$ and  from Lemma \ref{shifts01} the translation group $(T_\Bes(-i\b))_{\b\in\R}$ is strongly continuous on $\mathcal B$.   Hence the function $\mathbb R \ni \beta \mapsto r_{\alpha -i\beta}^{2} \in \Bq$ is continuous and bounded for each $\alpha >0$.   So the right-hand side of \eqref{B09} exists as a $\Bq$-valued Bochner integral.  Since the point evaluations are continuous in the $\mathcal B_0$-norm, the value of the right-hand side at $z \in \C_+$ is
\[
\int_\R \varphi(\b) r^2_{\a-i\b}(z) \, d\b = G_{\a,\varphi}(z).
\]

\noindent
2.  Let $z=x+iy$ and note that $G_{\alpha,\varphi}(\infty)=0$.  Since
\[
G'_{\alpha,\varphi}(z):=-2\int_\R \frac{\varphi(\beta)}{(z+\alpha-i\beta)^3}\,d\beta,\qquad z\in \C_{+},
\]
we have
\begin{align*}
|G'_{\alpha,\varphi}(x+iy)| &\le 2\|\varphi\|_{L^\infty}\int_\R \frac{d\beta}{((x+\alpha)^2+(\beta-y)^2)^{3/2}}\\
&\le \frac{2\|\varphi\|_{L^\infty}}{(x+\alpha)^2}\int_\R \frac{d\tau}{(1+\tau^2)^{3/2}} = \frac{4\|\varphi\|_{L^\infty}}{(x+\alpha)^2},
\end{align*}
so
\begin{align*}
\|G_{\alpha,\varphi}\|_{\Bq}
\le 4\|\varphi\|_{L^\infty}\int_0^\infty \frac{dx}{(x+\alpha)^2}
=\frac{4}{\alpha}\|\varphi\|_{L^\infty}. 
\end{align*}
Moreover,
\[
\|G_{\alpha,\varphi}\|_{L^\infty} \le \int_\R \frac{\|\varphi\|_\infty}{\a^2+\b^2} \, d\b = \frac{\pi}{\a} \|\varphi\|_{L^\infty}.
\]

\noindent
3.  It is clear that $Q$ is linear.  

Let $g \in \mathcal{W}$.  If $x>0$, then
\begin{align*}
|(Qg)(x+iy)| &\le\frac{2}{\pi}
\int_0^\infty \alpha \sup_{\beta\in \mathbb R}\,|g(\alpha+i\beta)|\,
\int_\R \frac{d\beta}{(x+\alpha)^2+(y-\beta)^2}\,d\alpha\\
&=2\int_0^\infty \frac{\alpha}{x+\alpha} \sup_{\beta\in \mathbb R}\,|g(\alpha+i\beta)|\,d\alpha\le 2\|g\|_{\mathcal W}.
\end{align*}
Hence
\begin{equation}\label{R1}
\|Qg\|_\infty\le 2\|g\|_{\mathcal W},
\end{equation}
and $(Qg)(\infty) = 0$ by the monotone convergence theorem.  Moreover,
\[
(Qg)'(z) = - \frac{2}{\pi} \int_0^\infty \a G'_{\a,\varphi_{\a,g}}(z) \,d\a, \qquad z \in \C_+,
\]
and
\[
\sup_{y\in\R} \left|(Qg)'(x+iy)\right|  \le \frac{2}{\pi} \int_0^\infty \a \sup_{y\in\R} \big| G'_{\a,\varphi_{\a,g}}(x+iy)\big| \,d\a.
\]
From this and \eqref{B11}, it follows that
\begin{align} \label{R2}
\|Qg\|_\Bq 
&\le \frac{2}{\pi} \int_0^\infty \a \, \int_0^\infty  \sup_{y\in\R} \big| G'_{\a,\varphi_{\a,g}}(x+iy)\big| dx \, d\a \\
&\le \frac{2}{\pi} \int_0^\infty \a\, \frac{4}{\a} \sup_{\b\in\R} |g(\a+i\b)|\,d\a = \frac{8}{\pi} \|g\|_\mathcal{W}.  \notag
\end{align}
By combining  (\ref{R1}) and (\ref{R2}) we obtain  (\ref{BC}).

Finally, for $f \in \Bq$, $f' \in \mathcal{W} \cap \operatorname{Hol}(\C_+)$ and $f = Q(f')$ by \eqref{repq}.
\end{proof}

\begin{rems} \label{L1L}
a) If $g(x+iy)= g_1(x) g_2(y)$, where $g_1 \in L^1(\mathbb R_+)$ and $g_2 \in L^\infty(\mathbb R)$,  then Proposition \ref{L1} implies that $Qg \in \mathcal B_0$.

\noindent b) Proposition \ref{L1} shows that $\Bq$ is isomorphic as a Banach space to the quotient $\mathcal{W}/\Ker Q$.  The argument in the proof of Proposition \ref{brim} shows that $\Ker Q$ contains $\overline{f}'$ for every $f \in \Bes$.    More generally, if $g_1 \in L^\infty(\R_+)$ and $g(x+iy) =g_1(x) \overline{f'(x+iy)}$, then $g \in \Ker Q$.    
\end{rems}

For $g \in \mathcal{W}$, we may approximate $g$ in various ways, and then we can lift those mappings to operators on $\Bes$ which approximate the identity map (modulo constant functions).    First, we consider cut-off operators $K_m, \, m\ge2$, on $\mathcal{W}$ defined by
\[
(K_mg)(z) := \begin{cases}  g(z),  &1/m \le \Re z \le m,  \\ 0, &\text{otherwise}.  \end{cases}
\]
Clearly $K_m$ is a contraction on $\mathcal{W}$, and $\lim_{m\to\infty} \|g-K_m g\|_{\mathcal{W}}=0$.   Then we consider the operator $\Kop_m \in L(\Bes,\Bq)$ defined by $\Kop_m f = Q(K_m f')$, so that
 \begin{align}\label{Def11}
(\Kop_m f)(z):&=-\frac{2}{\pi}\int_{1/m}^m \alpha\int_\R\frac{f'(\alpha+i\beta)}{(z+\alpha-i\beta)^2}\,d\beta\,d\alpha, \qquad z \in \C_+, \\
\|\Kop_m f\|_{\Bq} &\le \frac{8}{\pi} \|f\|_{\Bq},
\end{align}
using \eqref{wb0} and \eqref{BC}.   We give some alternative estimates.

\begin{prop}\label{Pr26}
For every $f\in \mathcal{B}$, and every $m\ge2$, the following hold:
\begin{equation}\label{Def12}
(\Kop_m\,f)(z)=\int_{2/m}^{2m} t f''(t+z)\,dt,
\end{equation}
\begin{equation}\label{Hest}
\|\Kop_mf\|_{\mathcal{B}_0}\le \frac{8\|f\|_\infty}{\pi}\log m,
\end{equation}
and
\begin{equation}\label{26A}
\|f-\Kop_mf\|_{\Bq}\le  \frac{8}{\pi}R_m(f),
\end{equation}
where
\begin{equation}\label{rmf}
{R}_m(f):=
\left\{\int_0^{1/m}+\int_m^\infty\right\} \sup_{\beta\in \R}\,|f'(\alpha+i\beta)|\,d\alpha \to 0,\qquad m\to\infty.
\end{equation}
\end{prop}

\begin{proof}
 Let $z=x+iy\in \C_{+}$.    By Cauchy's theorem,
\[
\int_\R \frac{f'(\alpha+i\beta)}{(z+\alpha-i\beta)^2}\,d\beta = \int_\R \frac{f'(\a+i\b)}{(2\a+z-(\a+i\b))^2} \,d\b =-2\pi f''(2\alpha+z),
\]
so
\[
(\Kop_m\,f)(z)=4\int_{1/m}^m \alpha f''(2\alpha+z)\,d\alpha=\int_{2/m}^{2m} t f''(t+z)\,dt.
\]

Using \eqref{f'est},
we have
\begin{align*}
|(\Kop_mf)'(z))| &\le
\frac{4}{\pi}\int_{1/m}^m\,\alpha \int_\R \frac{|f'(\alpha+i\beta)|}{|z+\alpha-i\beta|^3}\,d\beta\,d\alpha\\
&\le \frac{2\|f\|_\infty}{\pi}\int_{1/m}^m\, \int_\R \frac{d\beta}{|z+\alpha-i\beta|^3}\,d\alpha\\
&=
\frac{4\|f\|_\infty}{\pi}\int_{1/m}^m\,  \frac{d\alpha}{(x+\alpha)^2},
\end{align*}
so that
\[
\|\Kop_mf\|_{\mathcal{B}_0} \le \frac{4\|f\|_\infty}{\pi} \int_0^\infty \int_{1/m}^m \frac{d\a\,dx}{(x+\a)^2} 
= \frac{4\|f\|_\infty}{\pi}\int_{1/m}^m\,  \frac{d\alpha}{\alpha}=
\frac{8\|f\|_\infty}{\pi}\log m.
\]

From \eqref{repq} and Proposition \ref{L1}, we have
\[
\|f-\Kop_mf\|_\Bq = \|Q(f' - K_mf')\|_\Bq \le \frac{8}{\pi} \|f'-K_mf'\|_{\mathcal{W}} = \frac{8}{\pi} R_m(f). \qedhere
\]
\end{proof}

\begin{rem}
Using integration by parts, \eqref{Def12} may be written as 
\[
(\Kop_m f)(z) = f(a+z) - f(b+z) - af'(a+z) + bf'(b+z),
\]
where $a= 2/m, \, b=2m$.  In the notation of Lemma \ref{shifts01}, this gives
\begin{equation} \label{bhs}
\Kop_m  = T_\Bes(a)  - T_\Bes(b)  + a A_B T_\Bes(a) - bA_\Bes T_\Bes(b).
\end{equation}
As $a\to0+$ and $b\to\infty$, $T_\Bes(b)$ converges in the strong operator topology to the identity on $\Bq$ from the definition of $\|\cdot\|_\Bq$, and the other three terms converge to $0$, from Lemma \ref{shifts01}.     Thus $K_m^\vartriangle$ converges to the identity.   This is also a consequence of \eqref{26A} and \eqref{rmf}. 

In addition, it follows from \eqref{bhs} that any closed subspace of $\Bes$ which is invariant under horizontal shifts $T_\Bes(a), \, a\ge0$, is invariant under $\Kop_m$.
\end{rem}

Next we consider cut-off operators  $V_n, \, n\ge1$,  on $\mathcal{W}$ defined by
\[
(V_n g)(z) := \begin{cases}  g(z),  &|\Im z| \le n,  \\ 0, &\text{otherwise}.  \end{cases}
\]
Clearly $V_n$ is a contraction on $\mathcal{W}$.   We also consider the operators
 $QV_n : \mathcal{W} \to \Bq$, given by
 \begin{equation*}\label{Def11L}
(QV_n g)(z):=-\frac{2}{\pi}\int_0^\infty \alpha\int_{-n}^n\frac{g(\alpha+i\beta)}{(z+\alpha-i\beta)^2}\,d\beta\,d\alpha, \qquad z \in \C_+.
\end{equation*}

\begin{prop}\label{Pr27}
For every $g\in \mathcal{W}$ and $z \in \C_+$,
\begin{equation} \label{hstrip}
|(QV_n g) (z) - (Qg)(z)| \le \frac{8}{\pi} S_n(g),  \qquad n > 2|\Im z|, 
\end{equation}
where  
\[
S_n(g) := \int_0^\infty \frac{\alpha}{\alpha+n/2}\sup_{\beta\in \R}\,|g(\alpha+i\beta)|
\,d\alpha \to0,  \quad n \to \infty.
\]
\end{prop}

\begin{proof}
Let $z = x+iy \in \C_+$.   Then 
\begin{align*} 
\big |(Qg)(z)-(QV_n g)(z) \big| &\le \frac{2}{\pi}\int_0^\infty \alpha \int_{|\beta|>n}  \frac{|g(\alpha+i\beta)|}{|z+\alpha-i\beta|^2}\,d\beta\, d\alpha \\
&\le \frac{2}{\pi} \int_0^\infty \alpha\sup_{s\in \R}\,|g(\a+is)|
\int_{|\beta|>n} \frac{d\beta}{|z+\alpha-i\beta|^2}\,d\alpha.
\end{align*}
If $n > 2|y|$, we have
\begin{align*}
\int_{|\beta|>n}\frac{d\beta}{|z+\alpha-i\beta|^2}
&\le 2\int_n^\infty \frac{d\beta}{(x+\alpha)^2+(\beta-|y|)^2}\\
&\le 2\int_{n/2}^\infty \frac{d\beta}{\alpha^2+\beta^2}\le
4\int_{n/2}^\infty \frac{d\beta}{(\alpha+\beta)^2}=\frac{4}{\alpha+n/2}.
\end{align*}
It follows that \eqref{hstrip} holds.  By the monotone convergence theorem, $S_n(g) \to 0, \, n\to\infty$.
 \end{proof}

\section{Comparison of the algebras $\LT$ and $\Bes$}  \label{compspace}

The Hille-Phillips algebra $\mathcal{LM}$ is well understood and important in semigroup theory.  It is a subalgebra of $\Bes$ and a Banach algebra in the norm 
\[
\|\lt\mu\|_\HP := \|\mu\|_{M(\R_+)},
\]
Then $\|\lt\mu\|_\Bes \le 2\|\lt\mu\|_\HP$, but the norms are not equivalent (see Section \ref{compnorm}).  Thus estimates in terms of  $\|\cdot\|_\Bes$ are generally sharper than estimates in terms of $\|\cdot\|_\HP$, as well as being applicable to a larger class of functions.   On the other hand, $\mathcal{LM}$ is neither closed nor dense in $\mathcal B$ \cite[Lemma 2.14]{BGT}, and this leads to substantial complications in the theory of the $\mathcal B$-calculus.   Consequently it may be of value to identify the situations when a function from $\mathcal B$ can be approximated by functions from $\mathcal{LM}$ in a weaker sense, and to identify the closures of ${\mathcal LM}$ and its subspaces in $\mathcal B$, where practical.   Several results in these directions are given in this section, and they are potentially useful for applications of the $\mathcal B$-calculus to semigroup theory; see Theorem \ref{rajchman} for example.  

\subsection{Density of $\LT$ in other topologies}

In \cite[Lemma 2.14]{BGT} we showed that $\LT$ is dense in $\Bes$ in the topology of uniform convergence on compact subsets of $\C_+$.   We now show density in a slightly stronger sense,  using the approximating operators $K_n$ and $V_n$ from Propositions \ref{Pr26} and \ref{Pr27}.

\begin{prop} \label{qlm}
Let $g \in \mathcal{W}$ and assume that $g$ is supported by $[a,\infty)\times [-b,b]$, where $a>0$ and $b>0$.  Then $Qg \in \LT$ and
\[
\|Qg\|_\HP \le \frac{4b}{\pi a} \|g\|_{\mathcal{W}}.
\]
\end{prop}

\begin{proof}
Since
\begin{equation*}\label{RRA}
\frac{1}{(z+\alpha-i\beta)^2}=\int_0^\infty t e^{-(z+\alpha-i\beta)t}\,dt,\qquad z\in \C_{+},
\end{equation*}
 we have  $Qg = \lt h$, where
\[
h(t)= - \frac{2t}{\pi}\int_{a}^\infty \alpha e^{-\alpha t}\int_{-b}^b g(\alpha+i\beta)e^{i\beta t}\,d\beta\, d\alpha.
\]
Thus $Qg \in \LT$ and
\begin{align*}
\|Qg\|_\HP = \int_0^\infty |h(t)|\,dt
&\le \frac{2}{\pi}
\int_0^\infty t\int_a^\infty \alpha e^{-\alpha t}\int_{-b}^b |g(\alpha+i\beta)|\,d\beta\, d\alpha\,dt  \\
&\le \frac{4b}{\pi} \int_0^\infty t\int_{a}^\infty \alpha e^{-\alpha t}\sup_{s\in \R}\,|g(\alpha+is)|\, d\alpha\,dt
\\
&= \frac{4b}{\pi} \int_{a}^\infty \alpha^{-1}\sup_{s\in \R}\,|g(\alpha+is)|\, d\alpha \le \frac{4b}{\pi a} \|g\|_{\mathcal{W}}. \qedhere
\end{align*}
\end{proof}

For $f \in \Bes$ and $n\ge2$, define $\Qop_nf = Q V_n K_n f'$, so that
 \begin{equation*}\label{Def11M}
(\Qop_n f)(z):=-\frac{2}{\pi}\int_{1/n}^n \alpha\int_{-n}^n\frac{f'(\alpha+i\beta)}{(z+\alpha-i\beta)^2}\,d\beta\,d\alpha, \qquad z \in \C_+.
\end{equation*}

\begin{prop}\label{Pr28}
For every $f\in \mathcal{B}_0$, 
\begin{equation*}\label{HolA}
\Qop_n f\in \mathcal{LM},\qquad n\in \N,
\end{equation*}
and
\[
\lim_{n\to\infty}|(\Qop_n f) (z) - f(z)|=0,
\]
uniformly for all $z\in \overline{\C}_{+}$ with $|\Im z|\le c$,  for any $c>0$.
\end{prop}

\begin{proof}
Since $f' \in \mathcal{W}$ and $V_n K_n f'$ is supported by $[1/n,n] \times [-n,n]$, it is immediate from Proposition \ref{qlm} that $\Qop_n f \in \LT$.  By Proposition \ref{Pr27}, 
\[
\sup_{|\Im z|\le c} |(QK_n f')(z) - (\Qop_nf)(z)| \le \frac{8}{\pi} S_n(K_n f') \le \frac{8}{\pi} S_n(f'),  \quad n>2c.
\]
Since $QK_n f' - f \in \Bq$,  it follows from \eqref{eqnm} and \eqref{26A} that
\[
\|QK_n f' - f\|_\infty \le \|QK_n f' - f\|_\Bq  \le \frac{8}{\pi} R_n(f).
\]
Hence,
\[
\sup_{|\Im z|\le c} |f(z) - (\Qop_nf')(z)| \le \frac{8}{\pi} S_n(f') + \frac{8}{\pi} R_n(f), \qquad n>2c.
\]
The conclusion follows.
\end{proof}

\subsection{Closures of subspaces of $\LT$ in $\Bes$}

Let $L^1$, $M_d$ and $M_s$ stand for the subspaces of absolutely continuous, discrete, and singular non-atomic measures in $M(\mathbb R_+)$, respectively.  For each of these spaces $\mathcal N$, we let
\[
\mathcal{LN} = \{\mathcal{L}\mu : \mu \in \mathcal{N}\}.
\]
We show in this section that the closure of $\mathcal{LN}$ in $\Bes$ for the $\Bes$-norm is the same as the closure in the $H^\infty$-norm.   Note first that in each case $\mathcal{LN}$ is invariant under horisontal shifts, since if $f = \lt\mu$ then $T_\Bes(a)f = \lt\tilde\mu_a$, where $d\tilde\mu_a(t) = e^{-at}d\mu(t), a \ge 0$. We call a subspace $Y$ of $\mathcal B_0$ shift-invariant if
$T_\mathcal B(a)(Y)\subset Y$ for every $a \ge 0.$

For a subspace $Y$ of $\Bq$, let $\overline{Y}^\Bq$ denote the closure of $Y$ in $\Bq$, and $\overline{Y}^{H^\infty}$ denote the closure of $Y$ in $H^\infty(\mathbb C_+)$.   Then $\Bq \cap \overline{Y}^{H^\infty}$ is the closure of $Y$ in $\Bq$ with respect to the $H^\infty$-norm.

\begin{thm}\label{CCCAG}
Let $Y$ be a shift-invariant subspace of $\mathcal{B}_0$.   Then
\begin{equation} \label{GenD}
 \overline{Y}^{\mathcal{B}_0}= \Bq \cap \overline{Y}^{H^\infty}.
\end{equation}

Let $Z$ be a shift-invariant subspace of $\Bes$ containing the constant functions.  Then 
\[
 \overline{Z}^{\mathcal{B}}= \Bes \cap \overline{Z}^{H^\infty}.
\]
\end{thm}

\begin{proof}
Let $Y \subset \Bq$ be shift-invariant.  Then $\overline{Y}^\Bq$ is closed in $\Bq$ and shift-invariant, so it is invariant under $\Kop_n$ by \eqref{bhs}.   Fix $f\in \Bq \cap \overline{Y}^{H^\infty}$.   For every $n\in \N$ there exists 
$f_n\in Y$ such that
$
\|f_n-f\|_{\infty}\le 1/n.
$
Set
\[
\tilde{f}_n:=\Kop_n f_n \in \overline{Y}^\Bq. 
\]
By (\ref{Hest}) and (\ref{26A}), we infer that
\[
\|f-\tilde{f}_n\|_{\mathcal{B}_0}\le \|f-\Kop_nf\|_{\mathcal{B}_0}+\|\Kop_n(f-f_n)\|_{\mathcal{B}_0}
\le \frac{8}{\pi}{R}_n(f)+ \frac{8}{\pi n}\log n.
\]
Hence
\[
\lim_{n\to\infty}\,\|f-\tilde{f}_n\|_{\mathcal{B}_0}=0,
\]
and so $f \in \overline{Y}^\Bq$.  This proves \eqref{GenD}.

The statement about $Z$ follows from \eqref{GenD} for $Y = \Bq \cap Z$.
\end{proof}

Let $\mathcal R_0$ be the space of rational functions which vanish at infinity and have no poles in $\overline{\C}_+$.   Let $\mathcal {J}_0$ be the space of entire functions of exponential type which belong to $H^\infty(\mathbb C_+) \cap C_0(\overline{\C}_+)$.

\begin{thm}\label{CCCA}
Let $Y$ be any of the spaces $\LT$, $\mathcal{L}L^1$, $\LT_d$, $\LT_s$, $\mathcal{R}_0$ and $\mathcal{J}_0$.   Then the closure of $Y$ in $\Bes$ with respect to the $\Bes$-norm coincides with the closure of $Y$  in $\Bes$ with respect to the $H^\infty$-norm. 

If $Y$ is any of the spaces $\mathcal{L} L^1$, $\mathcal{R}_0$ and $\mathcal{J}_0$, then the closure of $Y$ in $\Bes$ is $\Bes \cap C_0(\overline{\C}_+)$.
\end{thm}

\begin{proof}
The first statement follows from Theorem \ref{CCCAG} since each space $Y$ is shift-invariant and either contains constants or is contained in $\Bq$.

It is shown in \cite[Appendix F, Proposition F.3]{HaaseB} that
\[
\overline {\mathcal R_0}^{H^\infty} 
= H^\infty(\mathbb C_+)\cap C_0(\overline{{\mathbb C}}_+),   
\]
Since the exponential functions $e_a$ for $a \in \C_+$ span a dense subset of $L^1$ and their Laplace transforms generate the algebra $\mathcal{R}_0$, so
\[ 
\overline {\mathcal{L}L^1}^{H^\infty} = \overline {\mathcal R_0}^{H^\infty}.
\]
Since functions of compact support are dense in $L^1$ and their Laplace transforms are in $\mathcal{J}_0$,
\[
\lt L^1 \subset \overline{\mathcal{J}_0}^{H^\infty} \subset  H^\infty(\mathbb C_+)\cap C_0(\overline{{\mathbb C}}_+),
\]
and then
\[
\overline{\mathcal{J}_0}^{H^\infty} = H^\infty(\mathbb C_+)\cap C_0(\overline{{\mathbb C}}_+).
\]
Now Theorem \ref{CCCA} implies all of the claims.
\end{proof}

\begin{rem}
The results in Theorem \ref{CCCAG} also hold when $Y$ is a subspace of $\Bq$ which is $\Kop_m$-invariant for some arbitrarily large $m\ge2$.

The spaces  $\LT$, $\mathcal{L}L^1$, $\LT_d$ and $\LT_s$ are all invariant under $\Kop_m$ for all $m\ge2$, as $\Kop_m\lt\mu    =  \lt \mu_m$, where
\begin{equation*} \label{M26}
\mu_m(dt):=t^2\left(\int_{2/m}^{2m} s  e^{-s t}ds\right)\,d\mu(t)=\left(\int_{2t/m}^{2mt} s  e^{-s}ds\right)\,d\mu(t).
\end{equation*}
\end{rem}

\section{Comparison of norms on $\LT$} \label{compnorm}

Recall that the $\mathcal B$-calculus coincides with the ${\rm HP}$-calculus on $\mathcal {LM}$.
Since the ${\rm HP}$-calculus plays a crucial  role in various applications of semigroup theory, 
for example in norm-estimates for functions of semigroup generators, it is important
to understand the advantages of the ${\mathcal B}$-calculus over the ${\rm HP}$-calculus arising from the different norms. To this aim, we provide precise asymptotics for $\|\varphi_s\|_{\rm HP}/\|\varphi_s\|_{\mathcal B}$
with respect to $s \to \infty$, where  $s$ is a real or integer parameter,
for several families $(\varphi_s) \subset \mathcal {LM}$.
Thus, the gap between the norm-estimates within the two calculi becomes visible.
The families $(\varphi_s)$ considered below are related to notable problems in semigroup theory.
A discussion of operator-theoretic aspects and motivations for their study
can be found in \cite{BGT}.

\subsection{\bf Cayley transforms} \label{Cayley1}

Let
\[
f_n(z)=\left(\frac{z-1}{z+1}\right)^n,\qquad z\in \C_{+},\quad n\in \N.
\]
In \cite[Lemma 3.7]{BGT} we obtained the following upper bound
\begin{equation}\label{upper}
\|f_n\|_{\mathcal{B}}\le 3+2\log(2n),\qquad n\in \N,
\end{equation}
by estimating a precise formula.   We now obtain the following lower bound.

\begin{lemma}\label{L111}
One has
\[
\|f_n\|_{\mathcal{B}}\ge 1+e^{-1}\log n,\qquad n\in \N.
\]
\end{lemma}

\begin{proof}
Near the end of the proof of \cite[Lemma 3.7]{BGT} we showed that
\[
\|f_n\|_{\mathcal{B}}=3+n\frac{(n-1)^{(n-1)/2}}{(n+1)^{(n+1)/2}}\log\frac{b_n}{a_n}-\left(\frac{1-a_n}{1+a_n}\right)^n
-\left(\frac{b_n-1}{b_n+1}\right)^n,
\]
where
\[
b_n=n+\sqrt{n^2-1},\qquad a_n=\frac{1}{b_n}=n-\sqrt{n^2-1}.
\]
Hence
\begin{align}\label{n=2}
\|f_n\|_{\mathcal{B}}&\ge 1+2n\frac{(n-1)^{(n-1)/2}}{(n+1)^{(n+1)/2}}\log b_n
\ge 1+2n\frac{(n-1)^{(n-1)/2}}{(n+1)^{(n+1)/2}}\log n\\
&\ge  1+\left(\frac{n-1}{n+1}\right)^{(n-1)/2}\log n.\notag
\end{align}
It is easily seen that 
\[
(1-1/x)^x\ge e^{-(1+1/x)},\qquad x\ge 2,
\]
so we also have 
\[
\left(\frac{n-1}{n+1}\right)^{(n-1)/2}=\left[\left(1-\frac{2}{n+1}\right)^{(n+1)/2}\right]^{\frac{n-1}{n+1}}
\ge e^{-(n+3)(n-1)/(n+1)^2}\ge e^{-1},
\]
for $n\ge3$.  Hence,
\begin{equation*}\label{lower}
\|f_n\|_{\mathcal{B}}\ge 1+e^{-1}\log n,\qquad n\ge 3.
\end{equation*}
Moreover, $\|f_1\|_\Bes = 3$, and $\|f_2\|_\Bes> 1 + e^{-1}\log 2$ from \eqref{n=2}. 
\end{proof}

Next we obtain two-sided estimates of $f_n$ in the Hille-Phillips norm.

\begin{lemma}\label{lemma_hp}
There are constants $c_0,c_1>0$ such that
\begin{equation*}
c_0 n^{1/2} \le \|f_n\|_{{\rm HP}} \le c_1 n^{1/2}, \qquad n \in \mathbb N.
\end{equation*}
\end{lemma}

\begin{proof}
As in \cite[Section 9]{Sasha},
\[
f_n(z)=1-2\int_0^\infty e^{-zt} e^{-t}L_{n-1}^{(1)}(2t)\,dt,
\]
where $L_n^{(1)}(\cdot)$ are the first order Laguerre polynomials.   Moreover
\[
c_0n^{1/2}\le \int_0^\infty  e^{-t/2}|L_{n-1}^{(1)}(t)|\,dt \le c_1'n^{1/2},\qquad n\in \N,
\]
for some constants $0<c_0<c_1'<\infty$ (see \cite{Askey} for details of this).   Thus
\[
\|f_n\|_{\text{HP}}=1+2\int_0^\infty  e^{-t}|L_{n-1}^{(1)}(2t)|\,dt=
1+\int_0^\infty  e^{-t/2}|L_{n-1}^{(1)}(t)|\,dt
\]
and
\[
1+c_0n^{1/2}\le \|f_n\|_{{\rm HP}}\le 1+c_1' n^{1/2},\qquad n\in \N. \qedhere
\]
\end{proof}

Combining \eqref{upper}, Lemma \ref{L111} and Lemma \ref{lemma_hp}, we get the following asymptotic relation.

\begin{cor}
One has
\[
 \frac{\|f_n\|_{{\rm HP}}}{\|f_n\|_{\mathcal{B}}}\asymp \frac{n^{1/2}}{\log n}\,, \quad n\to\infty.
\]
\end{cor}

\subsection{The exponent of a reciprocal}

Let 
\[
g_t(z)=e^{-t/(z+1)},\quad z\in \C_{+},\qquad t>0.
\]
By \cite[Lemma 3.4]{BGT}, 
\begin{equation} \label{gtb}
\|g_t\|_{\mathcal{B}} = \begin{cases} 2 - e^{-t},  &t \in (0,1],\\ 2 - e^{-1} + e^{-1} \log t, &t>1.  \end{cases}
\end{equation}

\begin{lemma}\label{hp2}
There are constants $C_1,C_2>0$ such that
\[
1+ C_1 t^{1/4} \le \|g_t\|_{{\rm HP}} \le 1+ C_2 t^{1/4}, \qquad t \ge 0.
\]
\end{lemma}

\begin{proof}
First, recall from \cite{deL} or \cite{Sasha} that 
\[
g_t(z)=1-\sqrt{t}\int_0^\infty e^{-zs}e^{-s}\frac{J_1(2\sqrt{ts})}{\sqrt{s}}\,ds,\qquad z\in \C_{+},
\]
where $J_1$ is the first-order Bessel function of the first kind.  By \cite[Section 4.8]{SpecialF}, we have
\[
|J_1(x)|\le \frac{c_1}{\sqrt{x}},\qquad x>0,
\]
and hence
\begin{align*}
\|g_t\|_{{\rm HP}} &=1+\sqrt{t}
\int_0^\infty e^{-s}\frac{|J_1(2\sqrt{ts})|}{\sqrt{s}}\,ds \\
&\le 1+\frac{c_1}{\sqrt{2}}t^{1/4}\int_0^\infty \frac{e^{-s}}{s^{3/4}}\,ds
= 1+\frac{c_1}{\sqrt{2}}\Gamma(1/4)t^{1/4}.
\end{align*}
On the other hand,
by \cite[Section 4.8]{SpecialF}, we have
\[
\int_0^t |J_1(s)|\,ds \ge c_0 \sqrt{t},\qquad t>0,
\]
and then
\begin{align*}
\sqrt{t}\int_0^\infty e^{-s}\frac{|J_1(2\sqrt{ts})|}{\sqrt{s}}\,ds&=
\int_0^\infty e^{-\tau/t}\frac{|J_1(2\sqrt{\tau})|}{\sqrt{\tau}}\,d\tau\
\ge e^{-1}\int_0^t\frac{|J_1(2\sqrt{\tau})|}{\sqrt{\tau}}\,d\tau\\
&= e^{-1}\int_0^{2\sqrt{t}}
|J_1(s)|\,ds\ge \sqrt2e^{-1}c_0 t^{1/4}, \qquad t>0.
\end{align*}
This completes the proof of the lemma.
\end{proof}

By combining \eqref{gtb} and Lemma \ref{hp2}, we obtain the next asymptotic relation.
\begin{cor}\label{two}
One has
\[
 \frac{\|g_t\|_{{\rm HP}}}{\|g_t\|_{\mathcal{B}}}\asymp\frac{t^{1/4}}{\log(t+1)},\qquad t\to \infty.
\]
\end{cor}

\subsection{Regularised exponent}

We consider a family of functions studied in \cite[Section III]{deL}.   Let 
\[
\varphi_t(z)=\frac{z}{z+1}e^{-t/z},\qquad z\in \C_{+}, \quad t >0.
\]

\begin{lemma}\label{HZLex}
One has
\[
e^{-1}\log(1+t/4) \le \|\varphi_t\|_\Bes\le   3+2e^{-1/2}\log(t+\sqrt{t^2+1}),\qquad t>0.
\]
\end{lemma}

\begin{proof}
Note first that $\|\varphi_t\|_\infty = 1$ and
\[
\varphi_t'(z) = \psi_{1,t}(z)+\psi_{2,t}(z),\qquad \psi_{1,t}(z)=\frac{e^{-t/z}}{(z+1)^2},
\quad \psi_{2,t}(z)=\frac{te^{-t/z}}{z(z+1)}.
\]
Moreover,
\[
\int_0^\infty \sup_{y\in \R}\,|\psi_{1,t}(x+iy)|\,dx\le \int_0^\infty \frac{dx}{(x+1)^2}=1.
\]
Using the  estimate $|\psi_{2,t}(x+iy)|\le t/x^2$, we obtain
\[
\int_t^\infty \sup_{y\in \R}\,|\psi_{2,t}(x+iy)|\,dx\le t\int_t^\infty \frac{dx}{x^2}=1.
\]
Next, 
\[
|\psi_{2,t}(x+iy)| =\frac{te^{-tx/(x^2+y^2)}}{\sqrt{(x^2+y^2)((x+1)^2+y^2)}}\le t  \rho_a(x^2+y^2),
\]
where
\[
a:=tx>0,\qquad  \rho_a(\tau):=\frac{e^{-a/\tau}}{\sqrt{\tau(\tau+1)}}, \quad \tau>0.
\]
We have
\begin{align*}
\rho_a'(\tau)&=\left(\frac{a}{\tau^{5/2}(\tau+1)^{1/2}}-\frac{1}{2\tau^{3/2}(\tau+1)^{1/2}}-
\frac{1}{2\tau^{1/2}(\tau+1)^{3/2}}\right) e^{-a/\tau}\\
&=\frac{2a(\tau+1)-\tau(\tau+1)-\tau^2}{2\tau^{5/2}(\tau+1)^{3/2}}e^{-a/\tau}
= -\frac{2\tau^2-\tau(2a-1)-2a}{\tau^{5/2}(\tau+1)^{3/2}} e^{-a/\tau}.
\end{align*}
Let $\tau_0=\tau_0(a)$ be the positive zero of the equation
$2\tau^2-\tau(2a-1)-2a=0$, so that
\begin{equation*}\label{zeroB}
\tau_0=\frac{(2a-1)+\sqrt{(2a-1)^2+16a}}{4}\in [a,2a].
\end{equation*}
Then
\begin{equation*}\label{zeroA}
\sup_{\tau>0}\,\rho_a(\tau)=\rho_a(\tau_0)
\le  \frac{e^{-1/2}}{\sqrt{a(a+1)}}.
\end{equation*}
Hence
\[
|\psi_{2,t}(x+iy)|\le  t \sup_{\tau>0}\,\rho_a(\tau)
\leq \frac{ e^{-1/2}\sqrt{t}}{\sqrt{x(tx+1)}}.
\]
Therefore,
\begin{align*}
\int_0^t \sup_{y\in \R}\,|\psi_{2,t}(x+iy)|\,dx &\le  e^{-1/2}\sqrt{t}\int_0^t
\frac{dx}{\sqrt{x(tx+1)}}\\
&= e^{-1/2}\int_0^{t^2}\frac{dx}{\sqrt{x(x+1)}}
=2e^{-1/2}\log\big(t+\sqrt{t^2+1}\big).
\end{align*}
Combining the estimates above, it follows that 
\[
\|\varphi_t\|_\Bes\le 3+2e^{-1/2}\log(t+\sqrt{t^2+1}),\qquad t> 0.
\]

To get the lower bound for $\|\varphi_t\|_{\mathcal B}$ it suffices to note that  
\[
|\varphi_t'(x+i\sqrt{tx})|=\frac{|(x+i\sqrt{tx})(1+t)+t|}{|x+i\sqrt{tx}|(1+x)^2+xt}e^{-t/(x+t)}
\ge \frac{e^{-1}t}{(1+x)^2+xt},
\]
and
\[
\int_0^\infty \frac{t}{(1+x)^2+xt} \,dx \ge  \int_1^\infty \frac{t}{x(4x+t)}\,dx=\log(1+t/4).  \qedhere
\]
\end{proof}

\begin{lemma}\label{hp3}  
There are constants $C_1,C_2>0$ such that
\[
C_1(1+t)\le \|\varphi_t\|_{\mathrm{HP}} \le C_2(1+t), \qquad t \ge 0. 
\]
\end{lemma}

\begin{proof}
Let 
\[
G(s):=\frac{J_1(2\sqrt{s})}{\sqrt{s}},\qquad s>0,
\]
where $J_1$ is the first-order Bessel function of the first kind.   Then, as in \cite[p.17]{deL} and \cite[Sections 3 and 4]{Gom}, 
\[
|G(s)|\le \frac{c}{1+s^{3/4}},\qquad G(0)=1,\quad G'\in L^1(\mathbb R_+),
\]
and 
\[
e^{-t/z}=1-t\int_0^\infty e^{-zs} G(ts)\,ds,\qquad z\in \C_{+}.
\]
Hence
\begin{equation}\label{AL}
ze^{-t/z}-z= -zt\int_0^\infty G(ts) e^{-zs} \,ds
=-t-t^2\int_0^\infty e^{-zs}G'(ts)\,ds,
\end{equation}
and
\[
t\int_0^\infty e^{-zs}G'(ts)\,ds=-1+(1-e^{-t/z})\frac{z}{t},\qquad z\in {\C}_{+}
\setminus \{0\}.
\]
Letting $z \to it/\pi$, 
it follows that
\[
t\int_0^\infty e^{-i t s/\pi}G'(ts)\,ds=-1+(1-e^{i\pi})\frac{i}{\pi}=-1+\frac{2i}{\pi},
\]
so
\[
\int_0^\infty |G'(s)|\,ds=t\int_0^\infty |G'(ts)|\,ds\ge 
t\left|\int_0^\infty e^{-i t s/\pi}G'(st)\,ds\right|
=\left|-1+\frac{2i}{\pi}\right|,
\]
and
\begin{equation}\label{zv}
\|G'\|_{L^1}\ge \sqrt{1+\frac{4}{\pi^2}}>1.
\end{equation}

From (\ref{AL}), the formula
\[
\frac{e^{-zs}}{1+z}= e^s \int_s^\infty e^{-z\tau} e^{-\tau}\,d\tau,
\]
and Fubini's theorem, we conclude that
\[
\varphi_t(z)=\frac{z}{z+1}-\frac{t}{z+1}-
t^2\int_0^\infty e^{-z\tau} e^{-\tau}\int_0^\tau e^{s} G'(ts)\,ds\,d\tau.
\]
Therefore,
\[
\varphi_t(z)=1-\int_0^\infty e^{-z\tau} K_t(\tau)\,d\tau,
\]
where
\[
K_t(\tau)=(1+t)e^{-\tau}+t^2e^{-\tau}\int_0^\tau e^{s} G'(ts)\,ds.
\]
Now
\begin{align*}
\|\varphi_t\|_{\text{HP}} &= 1+ \int_0^\infty |K_t(\tau)|\,d\tau  \\
&\le 2+t+t^2\int_0^\infty e^{-\tau}\int_0^\tau e^{s} |G'(ts)|\,ds\,d\tau\\
 &=2+t+t^2\int_0^\infty  |G'(ts)|\,ds\\
&=  2+(1+\|G'\|_{L^1})t,\qquad t>0.
\end{align*}

On the other hand,  
\begin{align*}
\|\varphi_t\|_{\text{HP}} = 1 +\int_0^\infty |K_t(\tau)|\,d\tau &\ge  1 + t^2\int_0^\infty e^{-\tau}\int_0^\tau e^{s} |G'(ts)|\,ds\,d\tau -(1+t)\\
 &=  c_1 t,\qquad t>0. 
\end{align*}
for  $ c_1=\|G'\|_{L^1}-1>0$, by \eqref{zv}.  Since $\|\varphi_t\|_{\text{HP}} \ge 1, \, t > 0$, it follows that $\|\varphi_t\|_{\text{HP}} \ge C_1(1+t), \, t>0$ for $C_1 = 1 - \|G'\|_{L^1}^{-1}$. 
\end{proof}

Combining Lemma \ref{HZLex} and Lemma \ref{hp3}, we get the following asymptotic relations.

\begin{cor}
One has
\[
\frac{\|\varphi_t\|_{{\rm HP}}} {\|\varphi_t\|_{\mathcal{B}}}\asymp \frac{t}{\log t},\qquad t\to\infty.
\]
\end{cor}

\section{Functional calculus: necessity and uniqueness}  \label{necuni}

The $\mathcal B$-calculus defined in \cite{BGT} is an efficient tool, and in this section we show that it  is optimal in two natural senses.  Firstly, the resolvent assumption in \eqref{8.1} is optimal, and secondly, the $\Bes$-calculus for a given operator $A$ is uniquely defined, so that it is necessarily given by the reproducing formula \eqref{fcdef}.
Thus the construction of the $\mathcal B$-calculus in \cite{BGT} appears to be essentially unique.

Let $A$ be a closed operator on a Banach space $X$, with dense domain.    Recall that $A$ is {\it sectorial of angle} $\theta \in [0,\pi/2]$ if $\sigma(A) \subset \overline \Sigma_\theta$ and
\begin{equation} \label{R01}
M_{A,\theta'} := \sup \left\{\|z(z+A)^{-1}\| :  z \in \Sigma_{\pi-\theta'}\right \} < \infty, \qquad \theta' \in (\theta,\pi/2].
\end{equation}
Moreover, $-A$ is the generator of a (sectorially) bounded holomorphic $C_0$-semigroup if and only if $A$ is sectorial of angle $\theta \in [0,\pi/2)$.  In that case we let
\begin{equation} \label{R02}
M_A := M_{A,\pi/2} = \sup \left\{\|z(z+A)^{-1} \right\| :  z \in \C_+\}.
\end{equation}
In this paper we adopt the common convention that a ``bounded holomorphic semigroup'' is sectorially bounded, i.e., bounded on a sector.

Following an established pattern for defining a functional calculus in operator theory, we shall say that an operator $A$ admits {\it a (bounded) $\Bes$-calculus} $\Phi$ if $A$ is densely defined, $\sigma(A) \subseteq \overline{\C}_+$ (so $A$ is closed), and there is a  bounded algebra homomorphism  $\Phi : \Bes \to L(X)$ such that  $\Phi(r_z) = (z+A)^{-1}$ for all $z \in \C_+$.  

We make the following observations about this definition:
\begin{enumerate}[(1)]
\item  It is not essential for the definition that $A$ is densely defined, but this is essential for many of the results that follow and it simplifies the presentation if it is a standing assumption.
\item  The functions $\{r_z : z \in \C_+\}$ form a resolvent family, so $\{\Phi(r_z): z \in \C_+\}$ form a pseudo-resolvent.  Hence if $\Phi(r_z) = (z+A)^{-1}$ for some $z \in \C_+$, then the same equation holds for all $z \in \C_+$.
\item  The existence of $\Phi$, the property (1) above and the fact that $\|r_z\|_\Bq = 1/\Re z$, imply that $\|(z+A)^{-1}\| \le \|\Phi\|/\Re z$.  Hence $A$ is sectorial of angle at most $\pi/2$.  Since $A$ is densely defined, $\lim_{a\in\R,a\to\infty} a(a+A)^{-1}x = x$ for all $x \in X$ \cite[Proposition 2.1.1c)]{HaaseB}.  From this it follows that the idempotent $\Phi(1)$ is the identity operator on $X$.
\end{enumerate}

\subsection{Necessity}
In \cite{BGT} the resolvent assumption \eqref{8.1} was shown to be sufficient for the construction of a $\mathcal B$-calculus.  Now we show that if $A$ admits a $\mathcal B$-calculus then \eqref{8.1} is satisfied,
so that the main result in \cite{BGT} is optimal as far as the resolvent assumptions are concerned.

\begin{thm} \label{necessary}
Let $A$ be an operator on $X$, and assume that $A$ admits a $\mathcal B$-calculus.  Then the resolvent assumption \eqref{8.1}  holds.   
\end{thm}

\begin{proof}
Let $\varphi$ be a continuous function on $\mathbb R$ with compact support, 
and let $G_{\alpha,\varphi}$ be given by \eqref{ksigma}. 
If $\Phi$ is a $\mathcal B$-calculus for $A$, then \eqref{B09} gives
\begin{equation*}
\Phi(G_{\alpha,\varphi})=\int_\R (\alpha -i\beta+A)^{-2}\varphi(\beta) \, d\beta,
\end{equation*}
and for all $x \in X$ and $x^*\in X^*$ such that $\|x\|=\|x^*\|=1$,
\begin{equation*}
\left|\int_\R \langle (\alpha -i\beta+A)^{-2}x, x^*\rangle \varphi(\b) \, d\beta \right| = |\langle \Phi(G_{\alpha,\varphi}) x, x^*\rangle| \le \|\Phi\|\, \|G_{\alpha,\varphi}\|_{\Bes},
\end{equation*}
since the calculus is bounded.  Now,  by Proposition \ref{L1},
\begin{align*}
\alpha\left|\int_\R \langle (\alpha-i\beta+A)^{-2}x,x^*\rangle \,\varphi(\beta)\,d\beta\right|
\le (4+\pi)\|\Phi\|\, \|\varphi\|_{L^\infty}
\end{align*}
for any $\varphi$ as above.   It then follows, from standard arguments around the Riesz representation theorem, that the estimate (\ref{8.1}) holds.
\end{proof}

\subsection{Uniqueness}
Next we consider the question of uniqueness for the $\Bes$-calculus.   The precise question is: 

 If an operator $A$ admits a $\mathcal B$-calculus $\Phi$, then  does the following hold
\begin{equation*}\label{bcalculus}
\langle \Phi (f) x,x^*\rangle=f(\infty)+ \frac{2}{\pi}\left \langle  g_{x, x^*}  , f \right \rangle_{\mathcal B}
\end{equation*}
for all  $f \in \Bes$, $x \in X$ and $x^* \in X^*$? 

\noindent  Here $g_{x.x^*}(z) = \langle (z+A)^{-1}x,x^*\rangle$, $g_{x,x^*} \in \Bov$, and $\langle \cdot,\cdot \rangle_\Bes$ is as in \eqref{dual}.

The following result shows that this question has a positive answer.

\begin{thm} \label{unique}
Let $A$ be an operator on $X$ and assume that $A$ admits a $\Bes$-calculus $\Phi$.   Then $\Phi$ is unique.
\end{thm}

\begin{proof}
From Theorem \ref{necessary} it follows that \eqref{8.1} holds, hence $-A$ is the generator of a
bounded $C_0$-semigroup $(T(t))_{t\ge 0}$ by either \cite[Theorem 1]{Gom} or \cite[Theorem 4.1 and its
proof]{Sf}.
Since the functions $e_a$ for $a>0$ span a dense subspace of $L^1(\R_+)$ and $\Phi(\lt e_a) = (a+A)^{-1}$, $\Phi$ is uniquely determined on $\lt L^1 = \{\lt g : g \in L^1(\R_+)\}$.    Let $f \in \ssp$, as defined in \eqref{ssp}, and $g \in L^2(0,\tau) \subseteq L^1(\R_+)$. By \cite[Lemma 2.13]{BGT}, $f \lt{g} \in \lt L^1$.   Thus $\Phi(\lt{g})$ and $\Phi(f)\Phi(\lt{g})$ are both determined uniquely.   It follows that $\Phi(f)$ is uniquely determined on the range of $\Phi(\lt{g})$.  In particular this implies that $\Phi(f)$ is uniquely determined on vectors of the form
$$
   \int_0^n e^{-t} T(t)x \, dt
$$
and hence on the range of $(1+A)^{-1}$.   Since $A$ is densely defined, it follows that $\Phi(f)$ is uniquely determined for all $f \in \ssp$.  By the density of $\ssp$ in $\Bq$, $\Phi$ is uniquely determined on $\Bq$.   Since $\Phi(1)$ is the identity operator, $\Phi$ is uniquely determined on $\Bes$.
\end{proof}

\begin{rem}\label{Bsg}
In the light of Theorems \ref{necessary} and \ref{unique} we may say that $A$ admits \emph{the} $\Bes$-calculus when a $\Bes$-calculus $\Phi$ for $A$ exists, or equivalently when \eqref{8.1} holds. 
 By [5,
Theorem 4.4],  we may then write $f(A)$ for $\Phi(f)$ and $e^{-tA}$ for $\Phi(e_t)$.   As noted in the
proof of Theorem \ref{unique}, $\left(e^{-tA}\right)_{t\ge0}$ forms a bounded $C_0$-semigroup with generator $-A$, as was established in \cite{Gom} and \cite{Sf}, in combination with \cite[Section 4.1]{BGT}.     Letting $K_A = \sup_{t\ge0} \|e^{-tA}\|$, the following standard Hille-Yosida estimate holds:
\begin{equation} \label{hy}
\|(\a+i\b+A)^{-n} \| \le K_A \a^{-n}, \qquad  \a>0,\, \b\in\R,\,n\ge 1. 
\end{equation}
\end{rem}

\subsection{Extensions} \label{s6.3}
Another important problem is whether the $\mathcal B$-calculus can be extended to a larger algebra than  $\Bes$ for all operators $A$ satisfying \eqref{8.1}.   Results of Peller \cite{Pel} in the discrete case, and White \cite{Whi} for bounded holomorphic semigroups, suggest that there may be such an extension, but it may be technically complicated.   A discussion of later advances in the discrete setting of  \cite{Pel} can be found in \cite[pp.\ 120-123]{Pisier}.

A standard approach to extending homomorphisms on commutative Banach algebras is to consider multiplier algebras.  
 For a commutative Banach algebra $\mathcal{A}$ with a bounded approximate identity $(e_n)$, let $\mathcal{M(A)}$ denote its Banach algebra of multipliers. 
If $H: \mathcal{A} \to L(X)$ is a homomorphism such that the set $\{H(a)x : a \in \mathcal{A}, x \in X\}$ is dense in $X$, then $H$ can be extended uniquely to a homomorphism $\tilde H: \mathcal{M(A)} \to L(X)$. 
The extension can be defined by the formula:
\begin{equation}\label{defmult}
\tilde H (M)x:=\lim_{n \to \infty}H(M e_n)x, \qquad M \in \mathcal{M (A)},
\end{equation}
where the limit on the right-hand side exists in the strong topology of $X$.
For details, see \cite[Theorem 2.4]{Choj} or \cite[Proposition 2.5]{Esterle}, 
both of which seem to be based on an idea from \cite{Johnson}.

The standard HP-calculus can be considered as an extension of the same calculus initially defined on $L^1(\mathbb R_+)$ and then extended to $M(\mathbb R_+)=\mathcal M(L^1(\mathbb R_+))$; see \cite[Theorem 3.2]{Choj}.   
Thus it is natural to try to identify the multiplier algebra of the Banach algebra $\mathcal B_0$
hoping to produce an extension $\tilde \Phi: \mathcal M(\Bq)\to L(X)$ of our $\mathcal B$-calculus.
Note that $\mathcal B_0$ is semisimple, and it is well-known that for a semisimple Banach algebra $\mathcal{A}$ any multiplier $M \in \mathcal{M(A)}$ can be identified with a multiplication operator
by a bounded continuous function in the Gelfand image of $\mathcal{A}$ (see  \cite[Theorem 1.2.2]{Larsen}).  Unfortunately, the maximal ideal space of $\mathcal B_0$ (and of much simpler algebras such as $H^\infty(\mathbb C_+)$) has a complicated structure, and thus we consider the following substitute for $\mathcal M(\mathcal B_0)$.
Let 
\begin{equation}\label{point}
\mathcal M_s(\mathcal B_0):=\{m \in \operatorname{Hol} (\C_+): m f \in \mathcal B_0 \,\, \text{for all}\,\, f \in \mathcal B_0\}
\end{equation} 
be the space of \emph{pointwise} multipliers of $\mathcal B_0$, a standard notion considered in the complex analysis literature (see \cite[Chapter 6]{Cima}, for example).
Using the Closed Graph Theorem, one infers that each $m \in \mathcal M_s(\Bq)$ defines
a bounded multiplication operator $M_m: f \mapsto mf$ on $\mathcal B_0$,
and $M_m \in \mathcal M(\mathcal B_0)$.
In what follows, we identify $m$ with $M_m$.
It is easy to show that $\mathcal M_s(\mathcal B_0)$ is a Banach subalgebra of $L(\mathcal B_0)$ with identity, containing $\mathcal B$.
Using \eqref{defmult} one can extend $\Phi$ from $\mathcal B_0$ to $\mathcal M_s(\mathcal B_0)$.   
If $\mathcal M_s(\mathcal B_0)$ strictly contains $\mathcal B$, then we would obtain a proper extension of the $\mathcal B$-calculus. 

The following statements of independent interest show that the above method for extension of Banach algebra homomorphisms 
cannot be applied to $\Phi :{\mathcal B}_0 \to L(X)$.  Firstly, $\mathcal B_0$ does not admit any approximate identities, let alone bounded ones.  Secondly, the algebra of pointwise multipliers of $\mathcal B_0$ is trivial in the sense that $\mathcal M_s(\mathcal B_0)=\mathcal B$.

\begin{thm}\label{App30}  {\rm1.}  
There exist no sequences $(f_n) \in \mathcal{B}_0$ such that
\[
\lim_{n\to\infty}\|f_ng - g\|_{\mathcal{B}_0}=0.
\]
for all $g \in \mathcal B_0$.  

\noindent {\rm2.} 
If $\mathcal M_s(\mathcal B_0)$ is defined by \eqref{point},
then $\mathcal M_s(\mathcal B_0)=\mathcal B$. 
\end{thm}

\begin{proof}
1.  
It suffices to show that $\|fe_1 - e_1\|_\Bq \ge 1$ for all $f \in \Bq$, where $e_1(z)=e^{-z}$. 

Suppose that $f \in \Bq$ and $\|fe_1 - e_1\|_\Bq = \delta \in (0,1)$, so that
\[
\int_0^\infty \sup_{\b\in \R}\,|1-f(\a+i\b)+f'(\a+i\b)|\,e^{-\a}\,d\a=\delta.
\]
Then there exists $\a_0\ge 0$ such that
\[
\sup_{\b\in \R}\,|1-f(\a_0+i\b)+f'(\a_0+i\b)|\le \delta,
\]
and by the maximum principle
\[
|1-f(z)+f'(z)|\le \delta,\qquad \Re z \ge \a_0.
\]
From this it follows that
\begin{equation*}\label{L30}
|f(z)-f'(z)|\ge 1-\delta,\qquad \Re z\ge \a_0.
\end{equation*}
However $f(\infty) = f'(\infty) = 0$, since $f \in \Bq$, so this is a contradiction, and the first statement is proved.

\noindent 2.
Clearly, if $f \in \mathcal B$ and $g \in \mathcal B_0$, then $fg \in \mathcal B_0$ 
since $\mathcal B_0$ is an ideal in $\Bes$. 
 Thus, it suffices to prove that $\mathcal M_s(\mathcal B_0)\subset \mathcal B$.
 
 If $M \in \mathcal M_s(\mathcal B_0)$, then $M$ is identified with a bounded multiplication
operator $M_m$ on $\mathcal B_0$ where $m$ is a holomorphic function on $\mathbb C_+$:
\[
M_m f(z)= m(z) f(z), \qquad f \in \mathcal B_0, \, z \in \mathbb C_+.
\]
Let $r_n(z):=(n+z)^{-1}, z \in\mathbb C_+,\, n\in\N$.   
Observe that $\|n r_n\|_{\mathcal B_0}=\|r_1\|_{\mathcal B_0} = 1,  \,n \in \N$, and 
\begin{equation}\label{fatou1}
\|n m r_n\|_{\mathcal B_0}\le \|M_m \|.
\end{equation}
Moreover, 
\begin{equation}\label{fatou2}
n m (z) r_n (z) \to m(z), \qquad n \to \infty,
\end{equation}
for every $z \in \mathbb C_+$.
Now using a result of Fatou type for $\mathcal B$ \cite[Lemma 2.3(1)]{BGT}, saying that if a pointwise limit of a bounded sequence from $\mathcal B$ exists on $\mathbb C_+$, then it belongs to $\mathcal B$, we infer from \eqref{fatou1} and \eqref{fatou2} that $m \in \mathcal B$.
\end{proof}

One might consider other subalgebras of $\Bes$.   For example, the Hardy-Sobolev algebra 
$H_1:=\{f \in \operatorname{Hol} (\C_+) : f' \in  H^1(\mathbb C_+) \}$
with the norm
\[
\|f\|_{H_1}:=\|f\|_{\infty}+\|f'\|_{H^1}, \qquad f \in H_1,
\]
 embeds continuously in $\mathcal B$ \cite[Lemma 2.4]{BGT}. 
Then $H^0_{1}:=\{f \in H_1: f(\infty)=0\}$ is an ideal in $H_1$ but not in $\Bes$, so its pointwise multiplier algebra $\mathcal{M}_s(H_1^0)$  (defined similarly to \eqref{point}) does not contain $\Bes$. 
Nevertheless, one may consider whether $\mathcal M_s(H^0_{1})$ contains any functions outside $\mathcal B$.   However, arguing in a similar way to Theorem \ref{App30}(2), one can show that $\mathcal M_s(H^0_{1})= H_{1}$.

\subsection {Generators of $C_0$-groups}
In order to understand the limitations of the $\mathcal B$-calculus, we will show that many generators of $C_0$-groups do not admit the $\mathcal B$-calculus.
 
Similarly to the abstract definition of the $\mathcal B$-calculus above, we say that a closed densely defined linear operator $A$ \emph{admits the $C_0(i\mathbb R)$-calculus} if $\sigma (A) \subset i\mathbb R$ and there exists a bounded algebra homomorphism $H: C_0(i\mathbb R)\to L(X)$ such that $H(r_z)=(z+A)^{-1}$ for every $z \in \mathbb C \setminus i\mathbb R$.   Since the closed linear span of  $\{r_z : z \in \mathbb C \setminus i\mathbb R \}$ is dense
in $C_0(i \mathbb R)$ (for example, by  the Stone-Weierstrass theorem), any such homomorphism $H$ is unique.

\begin{thm}\label{ccalc}
Let $-A$ be the generator of a bounded $C_0$-group on a Banach space $X$.
Then $A$ admits the $\mathcal B$-calculus if and only if $A$ admits the $C_0(i\mathbb R)$-calculus.
\end{thm}

\begin{proof}
Assume that $-A$ generates a bounded $C_0$-group $(T(t))_{t \in \mathbb R}$ on $X$.   Then by the spectral inclusion theorem for $C_0$-semigroups and the boundedness of $(T(t))_{t \in \mathbb R}$
we have $\sigma (A) \subset i\mathbb R$.   Let 
\[
\Delta (\alpha,\beta, A):=(\alpha +i\beta + A)^{-1}-(-\alpha+i\beta+ A)^{-1}, \qquad \alpha >0.
\] 
 
Assume in addition that $A$ admits the $\mathcal B$-calculus.   We will prove that there is a bounded homomorphism $H : C_0(\R) \to L(X)$ such that $H(\rho_z) = (z+A)^{-1}$ where $\rho_z(\b) = (z+i\b)^{-1}, \, z \in \C \setminus i\R$. 

By Theorem \ref{necessary},  \eqref{8.1} holds.
By \cite[Lemma 3.4]{Cojuhari} and the Closed Graph Theorem it follows that
\begin{equation}\label{groups}
\sup_{\alpha >0}  \int_{\mathbb R} |\langle \Delta(\a,\b,A)x, x^* \rangle| \, d\beta \le C_A
\end{equation}
for some $C_A > 0$ and  all $x \in X$,  $x^* \in X^*$ such that $\|x\|=\|x^*\|=1$. 

For a Schwartz function $g \in \mathcal S(\mathbb R)$, let $(\mathcal{F}g)(t) := \int_\R e^{-ist}g(s) \, ds$ be the Fourier transform of $g$.   By Parseval's identity, for all $\alpha >0$, 
\begin{equation}\label{identity_c}
\int_{\mathbb R} g(\beta) \Delta (\alpha,\beta, A)x \, d\beta =\int_{\mathbb R}e^{-\alpha |t|} (\mathcal F g) (t) T(t)x\, dt, \qquad x \in X.
\end{equation}
Define linear operators $A_\alpha : \mathcal S(\mathbb R) \to L(X),\alpha >0$, as
\begin{equation*}\label{aalpha}
A_\alpha (g) x: = \frac{1}{2\pi} \int_{\mathbb R} g(\beta) \Delta (\alpha,\beta, A) x \, d\beta, \quad x \in X,
\end{equation*}
and note that $\|A_\alpha (g)\|_{L(X)} \le \frac{C_A}{2\pi} \|g\|_{L^\infty}$ for every $\alpha >0$.    
By \eqref{identity_c},  the limit 
\begin{equation} \label{defh}
(H g)x:= \lim_{\alpha \to 0+}A_\alpha( g)x= \frac{1}{2\pi} \int_{\mathbb R} (\mathcal F  g) (t) T(t)x \, dt 
\end{equation}
exists in $X$, $Hg \in L(X)$, and 
\begin{equation}\label{cont}
\|H g\|_{L (X)}\le \frac{C_A}{2\pi} \|g\|_{L^\infty}, \qquad g \in \mathcal S(\mathbb R).
\end{equation}
Moreover, for all $g,h \in \mathcal S(\mathbb R)$, and $x \in X$,
\begin{align*}\label{homo}
H( gh)x &= \frac{1}{2\pi} \int_{\mathbb R}  \mathcal F(gh) (t) T(t)x\, dt = \frac{1}{4\pi^2} \int_{\mathbb R}  (\mathcal F(g)*\mathcal F(h)) (t) T(t)x\, dt  \\
&= \frac{1}{4\pi^2} \int_{\mathbb R}(\mathcal F g) (t) T(t)x \left( 
\int_{\mathbb R}(\mathcal F h) (s) T(s)x \, ds\right) \, dt
=H(g)H(h)x.\notag
\end{align*}

Since $\mathcal S(\mathbb R)$ is dense in $C_0(\mathbb R)$, the map $H$ extends by continuity to a bounded algebra homomorphism $H$ from $C_0(\mathbb R)$ to $L(X)$.    Replacing $g$ by $\mathcal{F}g$ in \eqref{defh},  we have 
\begin{equation} \label{neer} 
H(\mathcal{F}g)x =  \int_\R g(-t) T(t)x \, dt
\end{equation}
for $g \in \mathcal{S}(\R)$, and then by continuity for $g \in L^1(\R)$.   For $z \in \C_+$, taking $g(t) = e^{-tz}, \, t\ge0$ (extended by $0$ on $(-\infty,0)$), we obtain
\[
H(\rho_z)x = \int_\R e^{-zt} T(t)x \, dt = (z+A)^{-1}x,
\]
as required.  The case when $\Re z < 0$ is similar.

Now assume that $A$ admits the $C_0(i\R)$-calculus.   When $\varphi$ is continuous on $\R$ with compact support, the functions $G_{\alpha,\varphi}$, as in \eqref{ksigma}, are in $C_0(i\R)$.   Then the proof of Theorem \ref{necessary} can be applied to show that  \eqref{8.1} holds.
\end{proof}

There is an alternative way to show that a $C_0(i\R)$-calculus $H$ for a $C_0$-group generator induces a $\Bes$-calculus.   As discussed in Section \ref{s6.3}, $H$ can be extended to the multiplier algebra of $C_0(i\R)$, which is the algebra of bounded continuous functions on $i\R$.  Since $\Bes$ is continuously and injectively embedded in that algebra by $f \mapsto f^b$, one obtains a $\Bes$-calculus. 

\begin{remark}
For (negative) generators $A$ of $C_0$-groups, Theorem \ref{ccalc} shows that there are many cases when $A$ does not admit the $\mathcal B$-calculus. On the other hand, the operator $-A^2$ admits the $\mathcal B$-calculus as it is the negative generator of a bounded holomorphic $C_0$-semigroup on $X$ \cite[Corollary 3.7.15]{ABHN}.
\end{remark}

The fact that \eqref{groups} implies the existence of a bounded $C_0(i\mathbb R)$-calculus for $A$ was proved some time ago under various other assumptions  (see \cite[Theorem 3.6, Corollary 3.5]{deLaubenfels1} and references therein).  

Both \eqref{groups} and \eqref{neer} can be found explicitly in \cite[Lemma 2.4.3]{Neerven} and \cite[p.284]{Engel}.  For versions of \eqref{groups} and \eqref{neer} in a more general setting of vector-valued functions see [3, Theorem 4.8.1] and [37, Proposition 0.5].

Theorem \ref{ccalc} allows one to provide very simple examples of operators not admitting the $\mathcal B$-calculus.
The example below has essentially been given in \cite[Section 3.3]{Kantorovitz2}, see also \cite[Example 3.27]{deLaubenfels1}.

\begin{cor}\label{derivative}
Let $\mathcal -D$ be the generator of the $C_0$-group of right shifts on $X=L^p(\mathbb R), 1 \le p <\infty.$
Then $D$ admits the $\mathcal B$-calculus on $X$ if and only if $p=2.$
\end{cor}

\begin{proof}
Clearly, it is enough to prove the `only if' part.
Let $p \in [1,2), p \neq 2$ be fixed.
If $D$ admits the $\mathcal B$-calculus, then by Theorem \ref{ccalc} it admits $C_0(i\mathbb R)$-calculus.
Fix a non-zero $g \in  X$.  Using \eqref{cont} and \eqref{neer}, we infer that there exists $C >0$ such that
\begin{equation}\label{spectral}
\|g * f \|_{L^p(\mathbb R)}\le C \|\mathcal F f\|_{L^\infty}
\end{equation}
for all $f \in L^1(\mathbb R)$.   However, as remarked in \cite[p.252]{Kantorovitz2} 
(see also \cite[Example 3.27]{deLaubenfels1}), this contradicts \cite{Figa}.

If $2<p<\infty$ and $p'$ is the conjugate exponent, and $D$ admits the $\mathcal B$-calculus, then by Theorem \ref{necessary} its adjoint $D^*$ admits the $\mathcal B$-calculus on $L^{p'}(\mathbb R)$ as well. Since $D^*$ generates the $C_0$-group of left shifts on $L^{p'}(\mathbb R)$, which are similar to the right shifts under the change of variable $t \mapsto -t$, we again get a contradiction. 
\end{proof}

A discussion of issues around \eqref{spectral} and its various generalisations can be found in  \cite[Sections 4.6 and 5.5]{Larsen}.
 
An alternative approach to Corollary \ref{derivative} could rely on the fact that $D$ is not spectral on $L^p(\mathbb R)$ when $1 \le p <\infty, p \neq 2$  (see  \cite{Farwig}). One may then use the abstract results from \cite{deLaubenfels1}.
 Moreover, if $p \neq 2$ then all differential operators with constant coefficients on $L^p(\mathbb R^n)$ are not spectral, apart from the trivial case of the constant symbol (see \cite{Albrecht}).   The question when such operators generate $C_0$-semigroups has been studied in the literature (see \cite[Chapter 8]{ABHN}).  Thus, it is possible that other, more general, examples of differential operators not admitting the $\mathcal B$-calculus can also be provided.  

Note that if $A$ admits the $\mathcal B$-calculus, then $A-i\beta$ admits the $\mathcal B$-calculus for every $\beta \in \mathbb R$.  Thus the class of operators on Banach spaces admitting the $\mathcal B$-calculus contains vertical shifts of the generators of bounded holomorphic semigroups, and those operators do not generate bounded holomorphic $C_0$-semigroups.    
A complete description of operators with the $\mathcal B$-calculus is still out of reach.

\section{Spectral mapping theorems} \label{SMT}

In this section we show how the theory of the $\mathcal B$-calculus
allows us to obtain spectral mapping theorems for
parts of the classical HP-calculus given by Rajchman measures
and by absolutely continuous measures on $\R_+$. While the latter result is known
and not surprising, we are not aware of any spectral mapping theorems
for the Rajchman measures in the context of semigroup generators (the situation for \emph{groups} is essentially different).
Our results arise as direct applications of the $\mathcal B$-calculus, and we expect that more statements of a similar nature
can be proved using the $\mathcal B$-calculus theory.

Recall from Theorem  \ref{homomor_intro} the spectral inclusion theorem  for the $\mathcal B$-calculus proved in \cite[Theorem 4.17]{BGT}.  It shows that,  if $A$ admits the $\mathcal B$-calculus, then for every $f \in \mathcal B$ one has the inclusion
\begin{equation}\label{smp_b}
 f(\sigma(A)) \subset \sigma(f(A)),
\end{equation}
and this will be used in Corollary \ref{holex}.  In general, the inclusion may be far from being an equality, even for $f(z)=e^{-z}$.    Below, we identify certain $f \in \mathcal B$ for which the inclusion becomes (essentially) equality.

The following statement was given in \cite[Lemma 6.1]{Stafney} in a slightly less general form.    The same proof works in this case.   As in Theorem \ref{CCCA}, $\mathcal R_0$ denotes the space of rational functions which vanish at infinity and have no poles in $\overline{\C}_+$.

\begin{prop}\label{St_Lem}
Let $A$ be a densely defined operator with non-empty resolvent
set on a Banach space $X$.   If $(f_n)_{n\ge1} \subset \mathcal  R_0$ satisfies $f_n(A) \to B$ in $L(X)$ as $n \to \infty$, then $(f_n)_{n\ge1}$ converges uniformly on $\sigma(A) \cup \{\infty\}$ to a 
function $f \in C_0(\sigma(A))$ and
$$\sigma(B) \cup \{0\} = f(\sigma(A)) \cup  \{0\}.$$

\end{prop}

In \cite{Stafney}, the limit operators $B$ in Proposition \ref{St_Lem} were not identified.   We shall identify cases where $B= f(A)$ in the $\Bes$-calculus.   First we recall a spectral mapping theorem given in \cite[Theorem 8.2.11]{DaviesB}.  
In fact, the theorem is a direct corollary of Proposition \ref{St_Lem} and the density of the span of exponential functions in $L^1(\mathbb R_+)$.

\begin{prop}\label{davies}
Let $(T(t)_{t \ge 0}$ be a bounded $C_0$-semigroup on a Banach space $X$, with generator $-A$.  
If $g \in L^1(\mathbb R_+)$, $f = \mathcal{L}g$ and
$$
f(A)
:=\int_{0}^{\infty} g (t) T(t) \, dt,
$$
where the integral is defined in the strong operator topology, then
$$
\sigma(f (A)) \cup \{0\}= f (\sigma(A))\cup \{0\}.
$$
\end{prop}

In the version of Proposition \ref{davies} in \cite{DaviesB} it is also shown that if $A$ is unbounded, in addition to the assumptions above, then $0 \in \sigma(f(A))$.

In the next result we extend Proposition \ref{davies} for bounded  $C_0$-semigroups on Hilbert spaces.   A measure 
$\mu \in M(\mathbb R_+)$ is said to be a {\it Rajchman measure} if its Fourier transform belongs to $C_0(\mathbb R)$.     We denote the set of Rajchman measures on $\R_+$ by $M_0(\R_+)$.  They were studied intensively (sometimes under different names) in the 1970s and 1980s.   See  \cite{Kaniuth} and \cite{Lyons}, and the references therein.

\begin{thm}\label{rajchman}
Let $(e^{-tA})_{t \ge 0}$ be a bounded $C_0$-semigroup on a Hilbert space $X$.
If $f \in \mathcal B \cap C_0(\overline{\mathbb C}_+)$, then
\begin{equation} \label{smt}
\sigma(f(A)) \cup \{0\} = f (\sigma(A))\cup \{0\}.
\end{equation}
In particular,
if $\mu  \in M_0(\mathbb R_+)$
then
$$
\sigma(\mathcal L \mu (A)) \cup \{0\}=\mathcal L \mu (\sigma(A))\cup \{0\}.
$$
\end{thm}

\begin{proof}
If $f \in \mathcal B\cap C_0(\overline{\mathbb C}_+)$, then by Theorem \ref{CCCA} 
there exists $(f_n)_{ n \ge 1} \subset \mathcal R_0$ such that
\[
\lim_{n \to \infty} \|f_n-f\|_\Bes = 0.
\]
This implies that $f(A)$ is well-defined within the $\mathcal B$-calculus,
and
\[
\lim_{n \to \infty} \|f_n(A)-f(A)\|_{L(X)} = 0.
\]
By Stafney's Proposition \ref{St_Lem}, we have
\[
\sigma(f(A)) \cup \{0\}= f(\sigma(A))\cup \{0\}.
\]

If $f = \mathcal L\mu$ for $\mu \in M_0(\R_+)$, then $f \in \mathcal B \cap C_0(\overline{\mathbb C}_+)$, by the maximum principle, so the second claim follows.
\end{proof}

\begin{rems}
a)  Theorem \ref{rajchman} is also true if $A$ is an operator on a Banach space which admits the $\Bes$-calculus.  The proof is the same.   It would be interesting to know whether the result holds for all generators of bounded semigroups on Banach spaces even if $A$ does not admit the $\Bes$-calculus. 

\noindent b)
The proof of Theorem \ref{rajchman} relies essentially on the  use of spectral continuity in commutative Banach algebras in \cite{Stafney}.
A slightly neater but very similar argument was proposed in \cite[Theorem 2.1]{Dungey} in the framework of $l^1(\mathbb Z_+)$-calculus for power-bounded
operators on a Banach space. It was also based on the uniqueness of the limit of sets in the Hausdorff metric and it can probably be adjusted
to our purposes as well. The idea seems to originate from Kato's proof \cite[p.218]{Kato} of Wiener's theorem on invertibility of 
absolutely convergent Fourier series.   
\end{rems}

We illustrate Theorem \ref{rajchman} by a simple example. 

\begin{example}
 Let $(S_t)_{t\ge 0}$ be the $C_0$-semigroup of right shifts on $L^2(\mathbb R_+)$ with generator $-D$, and let $\mu \in M_0(\R_+)$.  Then by Plancherel's theorem,
\[
\sigma(\mathcal L \mu (D)) = \sigma (C_\mu)=\mathcal L\mu(\C_+)\cup \{0\},
\]
where $C_\mu (f):= \mu * f, f \in L^2(\mathbb R_+)$, is a convolution operator on $L^2(\mathbb R_+)$.
\end{example}

\section{Properties of the $\mathcal B$-calculus with respect to rescalings}

In this section we consider properties of the function $t \mapsto f(tA)$ for $f \in \Bes, \, t\ge0$.  In the special case when $f(z) = e^{-z}$, they become known properties of the $C_0$-semigroup generated by $-A$.   By Lemma \ref{shifts01}(7), the $\Bes$-norm is invariant under the rescalings $z \mapsto tz$, and correspondingly the rescalings are neutral for the $\Bes$-calculus in the sense that $\gamma_{tA} = \gamma_A$ where $\gamma_A$ is defined after \eqref{8.2}.   This enables us to establish some estimates which are uniform in $t$.

Note that the operator $f(tA)$ can be interpreted in two ways: either by applying the function $f$ to the operator $tA$, or by applying the function $f_t(z) = f(tz)$ to the operator $A$.   The outcomes coincide, as can be seen by a very simple change of variables in \eqref{fcdef}.

In some proofs in this section we use the validity of Plancherel's theorem for $L^2$-functions with values in a Hilbert space.   Suppose that $-A$ generates a bounded $C_0$-semigroup $(T(t))_{t\ge0}$ on a Hilbert space $X$, and let $K_A = \sup_{t\ge0} \|T(t)\|$.   For $x \in X$, Plancherel's theorem gives
\begin{equation} \label{plan}
\int_\R \left\|(\a+i\b+A)^{-1}x\right\|^2 \,d\b = 2\pi \int_0^\infty e^{-2\a t} \left\|T(t)x\right\|^2 \,dt \le \frac{\pi K_A^2}{\a}\|x\|^2.
\end{equation}
The same estimate holds with $A$ replaced $A^*$, by considering the dual semigroup.   Using the Hille--Yosida estimate \eqref{hy} and Cauchy-Schwarz, we obtain, for $n\ge2$ and $x,y \in X$,
\begin{align}  \label{gsf}
\lefteqn{\int_\R \left| \langle (\a+i\b+A)^{-n}x,y \rangle \right|\,d\b} \\
 &\le \int_\R \|(\a+i\b+A)^{-(n-1)}x\|\, \|(\a-i\b+A^*)^{-1}y\| \,d\b  \notag\\
&\le \frac{K_A}{\a^{n-2}}  \int_\R \|(\a+i\b+A)^{-1}x\|\, \|(\a-i\b+A^*)^{-1}y\| \,d\b \notag \\
 &\le \frac{\pi K_A^3}{\a^{n-1}} \|x\|\,\|y\|.  \notag
\end{align} 
When $n=2$, $K_A^3$ can be replaced by $K_A^2$, and the estimate shows that $A$ satisfies \eqref{8.1} and so admits the $\Bes$-calculus (see \cite[Example 4.1(1)]{BGT}).

Similarly Cauchy-Schwarz gives
\begin{multline} \label{cses}
\int_\R |\a-i\b|^{-1} \|(\a-i\b+A)^{-1}x\|\,d\b \\
 \le \left( \int_\R |\a-i\b|^{-2}\,d\b \int_\R \|(\a-i\b+A)^{-1}x\|^2 \,d\b\right)^{1/2} \le \frac{\pi K_A}{\a} \|x\|. 
\end{multline}

There is an estimate of the same form as \eqref{cses} when $A$ is sectorial of angle less than $\pi/2$ on a Banach space $X$.  Then 
\begin{equation} \label{stes}
\int_\R |\a-i\b|^{-1} \|(\a-i\b+A)^{-1}x\|\,d\b \le \frac{\pi M_A}{\a} \|x\|.
\end{equation}

\subsection{Convergence Lemma}
Some of our arguments will use the Convergence Lemma established in \cite[Corollary 4.14]{BGT}.    The proof in \cite{BGT} was formulated for operators, and that would suffice for our main results in this section, but first we establish a result which resembles \cite[Theorem 4.13]{BGT} but is formulated in the context of $\Bes$.

\begin{lemma}\label{conlem}
Let $(f_n)_{n\ge 1}\subset \mathcal{B}$ be such that
$
\sup_{n\ge 1}\,\|f_n\|_{\mathcal{B}}<\infty.
$
Assume that for every $z\in \C_{+}$ there exists
\begin{equation}\label{Conv21}
f_0(z):=\lim_{n\to\infty}\,f_n(z) \in \C,
\end{equation}
and for every $r>0$ one has
\begin{equation}\label{gz}
\lim_{\delta\to 0+}\,\int_0^\delta \sup_{|\beta|\le r}\,|f_n'(\alpha+i\beta)|\,d\alpha=0,
\end{equation}
uniformly in $n$. Let $g \in \Bes$ with $\lim_{|z|\to\infty} g(z)=0$,
and let
$
g_n(z):=f_n(z)g(z), n\ge 0.
$
Then $f_0\in \mathcal{B}$, $g_n\in \mathcal{B}_0, n \ge 0$, and
\[
\lim_{n\to\infty}\,\|g_n-g_0\|_{\mathcal{B}}=0.
\]
\end{lemma}

\begin{proof}    By \cite[Proposition 2.3(1)]{BGT},  $f_0\in \mathcal{B}$.   Replacing $f_n$ by $f_n-f_0$, we may assume that $f_0=0$.   Moreover  $g_n\in \mathcal{B}_0, n \ge 0$, since $g \in \Bq$.    We have
\begin{align*}
\|g_n\|_{\mathcal{B}_0} &\le \int_0^\infty \sup_{\Re z = \a}\,|f_n'(z)g(z)|\,d\alpha
+\int_0^\infty \sup_{\Re z=\a}\,|f_n(z)g'(z)|\,d\alpha\\
&=: I_n+J_n.
\end{align*}

Let $C = \sup_{n\ge1} \|f_n\|_\infty$.   For fixed $\delta\in (0,1)$ and $r \ge 0$,
\begin{align*}
J_n &\le C  \left(\int_0^\delta+\int_{1/\delta}^\infty\right)\sup_{\beta\in\R}\,|g'(\alpha+i\beta)|\,d\alpha + C \int_\delta^{1/\delta} \sup_{|\beta|\ge r}\,|g'(\alpha+i\beta)|\,d\alpha\\
&\null\hskip50pt+\sup_{\Re z\ge \delta}\,|g'(z)| \int_\delta^{1/\delta} \sup_{|\beta|\le r}\,|f_n(\alpha+i\beta)|\,d\alpha.
\end{align*}
Now let $n\to\infty$.  Since $f_n(z)\to 0$ as $n\to\infty$ for any $z\in \C_{+}$, Vitali's theorem 
implies that the final integral tends to zero, so
\[
\limsup_{n\to\infty} J_n \le C  \left(\int_0^\delta+\int_{1/\delta}^\infty\right)\sup_{\beta\in\R}\,|g'(\alpha+i\beta)|\,d\alpha + C \int_\delta^{1/\delta} \sup_{|\beta|\ge r}\,|g'(\alpha+i\beta)|\,d\alpha.
\]
Next let $r \to \infty$.   Since  $g'(z)\to 0$ as $|z|\to\infty, \Re z \ge \delta$, the final integral tends to zero.   Hence
\[
\limsup_{n\to\infty} J_n \le C  \left(\int_0^\delta+\int_{1/\delta}^\infty\right)\sup_{\beta\in\R}\,|g'(\alpha+i\beta)|\,d\alpha.
\] 
Finally let $\delta\to0+$.   Since $g \in \Bes$, the two integrals tend to zero.  Hence $\lim_{n\to\infty} J_n = 0$.

Next, let $r \ge 1$ and $\delta \in (0,1)$.   We have
\begin{align*}
I_n &\le \|g\|_\infty \int_0^r  \sup_{|\beta|\le r}\,|f_n'(\alpha+i\beta)|\,d\alpha
 + \sup_{|{\rm Im}\,z|\ge r}\,|g(z)|\,\int_0^r \sup_{|\beta|\ge r}|f_n'(\alpha+i\beta)|\,d\alpha \\
&\null\hskip50pt +\sup_{{\rm Re}\,z\ge r}\,|g(z)|\,\int_r^\infty \sup_{\beta\in \R}|f_n'(\alpha+i\beta)|\,d\alpha\\
&\le\|g\|_\infty \sup_{k\ge1} \int_0^\delta \sup_{|\beta|\le r}|f_k'(\alpha+i\beta)|\,d\alpha +  \|g\|_\infty \int_\delta^r \sup_{|\beta|\le r}|f_n'(\alpha+i\beta)|\,d\alpha \\
&\null \hskip50 pt +\sup_{|z|\ge r}\,|g(z)|\,\sup_{k\ge 1}\|f_k\|_{\mathcal{B}}.
\end{align*}
Let $n\to\infty$.   By Vitali's theorem applied to $(f_n')$, the second term tends to zero, so
\[
\limsup_{n\to\infty} I_n \le \|g\|_\infty \sup_{k\ge1} \int_0^\delta \sup_{|\beta|\le r}\,|f_k'(\alpha+i\beta)| \,d\alpha + \sup_{|z|\ge r}\,|g(z)|\,\sup_{k\ge 1}\|f_k\|_{\mathcal{B}}.
\]
Now let $\delta\to0+$.   By \eqref{gz}, the first term tends to zero.   Then let $r \to \infty$, and the second term tends to zero.   Hence $\lim_{n\to \infty} I_n = 0$.  
\end{proof}

\begin{remark}
In \cite[Theorem 4.13]{BGT}, we assumed that $g' \in H^1(\C_+)$ and $g(\infty) = 0$.   By \cite[Proposition 2.4]{BGT}, this implies the assumption in Lemma \ref{conlem} that $g \in \Bes$ and $\lim_{|z|\to\infty} g(z) = 0$.  Moreover, the function $g(z) = (1+z)^{-1} e^{-z}$ satisfies the assumption in Lemma \ref{conlem}, but $g' \notin H^1(\C_+)$.
\end{remark}

\begin{cor} \label{sotc}
Let $A$ be an operator which admits the $\Bes$-calculus, and let $f_n$ and $f_0$ be as in Lemma \ref{conlem}.   Then $f_n(A) \to f_0(A)$ in the strong operator topology as $n\to\infty$.   

If $A \in L(X)$, then the convergence is in operator-norm.
\end{cor}

\begin{proof}
Let $g(z) = (1+z)^{-1}$.   By Lemma \ref{conlem} and continuity of the $\Bes$-calculus,
\[
g_n(A) = f_n(A) (1+A)^{-1}  \to f_0(A)(1+A)^{-1}
\]
in operator-norm as $n\to\infty$.   Since $(f_n(A))_{n\ge1}$ is uniformly bounded and $D(A)$ is dense, the first statement follows.   If $A \in L(X)$, then $f_n(A) =  g_n(A)(1+ \nolinebreak[4] A) \to f_0(A)$ in operator-norm as $n\to\infty$.
\end{proof}

\subsection{Continuity}

Now we apply the Convergence Lemma to the function $t \mapsto f(tA)$.

\begin{thm} \label{ftacont}
Let $A$ be an operator which admits the $\Bes$-calculus, and let $f \in \Bes$.   Then $t \mapsto f(tA)$ is continuous from $\R_+$ to $L(X)$ with the strong operator topology.

If $A \in L(X)$, then the map is continuous with respect to the norm topology on $L(X)$.
\end{thm}

\begin{proof}
We consider the functions $f_t(z) = f(tz), \, t>0$.   Then $f_t \in \Bes$ and $\|f_t\|_\Bes = \|f\|_\Bes$, and $t \mapsto f_t(z)$ is continuous for each $z \in \C_+$.   

Let $r>0$, $\tau \in \R_+$,  $t \in [0,\tau+1]$ and $\delta>0$.   Then
\begin{align*}
\int_0^\delta \sup_{|\beta|\le r} |f_t'(\a+i\b)| \, d\a &\le \int_0^\delta t \sup_{\b\in\R} |f'(t(\a+i\b))| \,d\a \\
& \le \int_0^{\delta(\tau+1)} \sup_{\b\in\R} |f'(\a+i\b)| \,d\a \to 0,
\end{align*}
uniformly for $t \in [0,\tau+1]$, as $\delta\to0+$.  

 It follows that, for every sequence $t_n$ converging to $\tau$ in $\R_+$, the functions $(f_{t_n})_{n\ge1}$ satisfy the assumptions of Lemma \ref{conlem} with $f_\tau$  as the limit function, so Corollary \ref{sotc} implies both statements.
\end{proof}

If  $A$ is sectorial of angle $\theta \in [0,\pi/2)$, and $z \in \Sigma_{\pi/2-\theta}$, then $zA$ is sectorial of angle less than $\pi/2$, and  $f(zA)$ is defined for all $f \in \Bes$, by the $\Bes$-calculus.   We can then obtain a stronger result than in Theorem \ref{ftacont}.

\begin{prop} \label{holhol}
Let $A$ be a sectorial operator of angle $\theta \in [0,\pi/2)$, and let $f \in\Bes$.   Then $z \mapsto f(zA)$ is holomorphic on $\Sigma_{\pi/2-\theta}$.   In particular, $t \mapsto f(tA)$ is $C^\infty$ on $(0,\infty)$.   
\end{prop}

\begin{proof}
For $\theta' \in (\theta,\pi/2)$ and  $z \in \Sigma_{\theta'-\theta}$, we have, from \cite[Corollary 4.11]{BGT},
\[
f(zA) = f(\infty) - \frac{2}{\pi}  \int_0^\infty   \int_{\R} \alpha (\alpha -i\beta +zA)^{-2}  {f'(\alpha +i\beta)} \, d\beta\,d\alpha.
\]
The $z$-derivative of the integrand  is 
\begin{multline*}
2\alpha A (\a-i\b+zA)^{-3} f'(\a+i\b) \\
= \frac{2\a}{z} \left( (\a-i\b+zA)^{-2} - (\a-i\b) (\a-i\b+zA)^{-3} \right) f'(\a+i\b).
\end{multline*}
For $z \in\Sigma_{\theta'-\theta}$ with $|z|>\delta>0$, the norm of this is dominated by 
\[
\frac{4\a}{\delta} M^3_{A,\theta'} (\a^2+\b^2)^{-1} \sup_{s\in\R} |f'(\a+is)|,
\]
which  is integrable over $\C_+$.   This proves the claim.
\end{proof}

If a $C_0$-semigroup with generator $-A$ is norm-continuous on $(0,\infty)$, then the Riemann--Lebesgue lemma implies that 
\begin{equation*}
\lim_{|\beta|\to\infty}\,\|(\alpha+i\beta+A)^{-1}\|=0,
\end{equation*}
for some $\a\in\R$.   A longstanding open question was whether the converse holds.   On Hilbert spaces, a positive answer was initially given in \cite{You}, followed by simpler proofs and related results in \cite{Elm1} and \cite{Goer}.   Negative answers on general Banach spaces were given in \cite{Chill} and \cite{Matrai}.    See also \cite[Section 3.13]{ABHN}.  

In Theorem \ref{Norm1} we will place the norm-continuity theorem for semigroups on Hilbert spaces in the framework of our $\mathcal B$-calculus.   The result is very much in the spirit of Section \ref{SMT}, where spectral mapping theorems for semigroups were extended to spectral mapping theorems for functions from $\mathcal B$.  For the proof we need the following result establishing compatibility between the approximating operators in Section \ref{approx} and the $\Bes$-calculus.

\begin{lemma} \label{KB}
Suppose that $A$ admits the $\Bes$-calculus on a Banach space $X$, let $f \in \Bes$, and $\Kop_m f$ be as defined in \eqref{Def11}.  Then, for $x \in X$ and $x^* \in X^*$,
\[
\langle (\Kop_mf)(A)x,x^* \rangle = - \frac{2}{\pi} \int_{1/m}^m \a \int_\R \langle (\a-i\b+A)^{-2}x,x^* \rangle f'(\a+i\b)\,d\b\,d\a.
\]
\end{lemma}

\begin{proof}
Using \eqref{fcdef} and differentiating \eqref{Def11}, we have
\begin{align*}
\lefteqn{\langle (\Kop_m f)(A)x,x^* \rangle} \\
 &= - \frac{2}{\pi} \int_0^\infty t \int_\R   \langle (t-is+A)^{-2}x,x^* \rangle  (\Kop_m f)'(t+is) \,ds\,dt \\
&=  - \frac{4}{\pi^2} \int_0^\infty  \int_\R t \langle  (t-is+A)^{-2}x,x^* \rangle \\
& \null\hskip50pt \int_{-1/m}^m \int_\R \a (\a-i\b+t+is)^{-3} f'(\a+i\b) \,d\b\,d\a \,ds\,dt \\
&=-\frac{4}{\pi^2} \int_{1/m}^m \int_\R \a  f'(\a+i\b) \\
& \null \hskip50pt \int_0^\infty\int_\R t \langle  (t-is+A)^{-2}x,x^* \rangle (\a-i\b+t+is)^{-3} \,ds\,dt  \,d\b\,d\a  \\
&= -\frac{2}{\pi} \int_{1/m}^m \a \int_\R \langle r_{\a-i\b}^2(A) x,x^* \rangle   f'(\a+i\b) \,d\b\,d\a \\
&= - \frac{2}{\pi} \int_{1/m}^m \a \int_\R \langle (\a-i\b+A)^{-2}x,x^* \rangle f'(\a+i\b)\,d\b\,d\a. \qedhere
\end{align*}
\end{proof}

\begin{rem}
In Lemma \ref{KB} the cut-off operator $\Kop_m$ on $\mathcal{W}$ corresponding to the vertical strip $\{z : m^{-1} \le \Re z \le m\}$ can be replaced by any $S \in L(\mathcal W)$.   Let $S^\vartriangle \in L(\Bes)$ be defined by $S^\vartriangle f = Q (Sf')$.   Then
\[
\langle (S^\vartriangle f)(A)x,x^* \rangle = - \frac{2}{\pi} \int_0^\infty \a \int_\R \langle (\a-i\b+A)^{-2}x,x^* \rangle (Sf')(\a+i\b)\,d\b\,d\a.
\]
This can be applied to the cut-off operators considered in Propositions \ref{Pr26} and \ref{Pr28}.
\end{rem}

\begin{thm}\label{Norm1}
Let $-A$ be the generator of a bounded $C_0$-semigroup on a Hilbert space $X$, and
let $f \in \mathcal B$. 
 Suppose that for some $\alpha>0$,
\begin{equation}\label{someaa}
\lim_{|\beta|\to\infty}\,\|f(\alpha+i\beta)(\alpha-i\beta+A)^{-1}\|=0.
\end{equation}
Then $f(\cdot A) \in C((0,\infty), L(X))$.

In particular,  $f(\cdot A) \in C((0,\infty), L(X))$ for all $f \in \Bes$ if
\begin{equation}\label{somea}
\lim_{|\beta|\to\infty}\,\|(\alpha-i\beta+A)^{-1}\|=0.
\end{equation} 
\end{thm}

\begin{proof}
We may assume that $f(\infty) = 0$.  Let
\[
f_t(z):= f(tz), \qquad t >0, z \in \mathbb C_+.
\]
Let $\tau>0$ be fixed.  We need to prove that
\begin{equation}\label{LimitN}
\lim_{t \to \tau}\, \|f_t(A)-f_\tau(A)\|=0.
\end{equation}
We consider the approximating operators $\Kop_m$
for $m>1$, as in \eqref{Def11}.
From Lemma \ref{KB} and Proposition \ref{Pr26} we obtain
\[
\big\|f_t(A)-(\Kop_m f_t)(A)\big\|\le  \frac{8}{\pi} \left\{\int_0^{t/m}+\int_{tm}^\infty\right\}\, \sup_{\beta\in \R}\,|f'(\alpha+i\beta)|\,d\alpha \to 0
\]
as $m \to \infty$, uniformly for $t$ in compact subsets of $(0,\infty)$.   
Thus, it suffices to prove that for fixed $m>1$ and $\tau>0$,
\begin{equation}\label{ABN}
\lim_{t \to \tau}\,\big\|\big(\Kop_m(f_t-f_\tau)\big)(A)\big\|=0.
\end{equation}

Take $x,y  \in X$ with $\|x\|\le1$ and $\|y\|\le1$, and take $n\ge1$ and $t \in [\tau/2,2\tau]$.  Using integration by parts, which is justified by finiteness of the integral in \eqref{gsf} for $n=2$, we obtain
\begin{align*}
\lefteqn{\langle (\Kop_m(f_t-f_\tau))(A)x,y\rangle} \\
&=-\frac{2}{\pi}\int_{1/m}^m \alpha \int_\R \langle (\alpha-i\beta+A)^{-2}x,y\rangle
(f_t'(\alpha+i\beta)-f_\tau'(\alpha+i\beta))\,d\beta\,d\alpha\\
&=\frac{4}{\pi}\int_{1/m}^m \alpha \int_{-n}^n \langle (\alpha-i\beta+A)^{-3}x,y\rangle
(f_t(\alpha+i\beta)- f_\tau(\alpha+i\beta))\,d\beta\,d\alpha\\
&\null\hskip15pt +\frac{4}{\pi}\int_{1/m}^m \alpha \int_{|\beta|\ge n}\,\langle (\alpha-i\beta+A)^{-3}x,y\rangle
(f_t(\alpha+i\beta)- f_\tau(\alpha+i\beta))\,d\beta\,d\alpha \\
&=: I'_{mn} + J'_{mn}. 
\end{align*}
 We will treat $I'_{mn}$ and $J'_{mn}$ separately.

We start with $J'_{mn}$.  Note that both the resolvent $(z +A)^{-1}$ and $f(z)$ are bounded in each half-plane $\RR_\ep, \, \ep >0$.  Using \eqref{someaa} and applying Vitali's theorem to the family of holomorphic functions $\{ \overline{f(\overline{z}}) (z \pm in + A)^{-1}: n \ge 1\}$ 
on the rectangle $\{(\alpha +i\beta: \alpha \in (a,b), \beta \in (-1,1)\}$ we conclude  
that  
\[
\|f(\alpha +i\beta) (\alpha -i\beta +A)^{-1}\|\to 0, \qquad |\b| \to \infty,
\]
uniformly for $\alpha \in [a,b]$ for any $b>a>0$.  

Let $t_0:=\min(\tau/2,1)$ and $t_1:=\max(2\tau,1)$, and 
\[
c_{mn} :=  \sup \Big\{\|f(\alpha+i\beta)(\alpha-i\beta+A)^{-1}\|  : \alpha\in [t_0/m,t_1 m],\,|\beta|\ge t_0 n \Big\}.
\]
Applying the paragraph above with $a=t_0/m, \, b=t_1 m$, we conclude that $\lim_{n\to\infty} c_{mn} = 0$.
Using \eqref{gsf} for $A$ and $A^*$, we obtain
\begin{align*}
 |J'_{mn}| &\le \frac{8c_{mn}}{\pi}\int_{1/m}^m
\alpha\int_{R}\|(\alpha-i\beta+A)^{-1}x\|
\|(\alpha+i\beta+A^{*})^{-1}y\|\,d\beta \, d\alpha\\
&\le
\frac{4c_{mn}}{\pi}\int_{1/m}^m
\alpha \int_{\mathbb R} \left( \|(\alpha-i\beta+A)^{-1}x\|^2+
\|(\alpha+i\beta+A^{*})^{-1}y\|^2 \right) \,d\beta\, d\alpha
\\
&\le 8K_A^2 c_{mn} m
\end{align*}
%\begin{align*}
%|J'_{mn}| &\le \frac{8 c_{mn}}{\pi} \int_{1/m}^m \alpha \int_\R\,\left|\langle \alpha-i\beta+A)^{-3}x,y  %\rangle\right|\,d\beta\,d\alpha\\
%&\le \frac{8 c_{mn}}{\pi} \int_{-1/m}^m \frac{\pi K_A^3}{\a} \,d\a \\
%&= 16 K_A^3 c_{mn}  \log m \to 0, \qquad n\to\infty,
%\end{align*}
uniformly for $t \in [\tau/2,2\tau]$, $\|x\|\le1$ and $\|y\|\le1$.

Now it suffices to show that, for fixed $m,n$, $I'_{mn} \to 0$, uniformly for $\|x\|\le1$ and $\|y\| \le 1$,  as $t \to \tau$.  Using \eqref{hy} for $n=3$ and the bounded convergence theorem,
\begin{multline*}
|I'_{mn}| \le K_A \int_{1/m}^m \alpha^{-2} \int_{-n}^n 
|f_t(\alpha+i\beta)- f_\tau(\alpha+i\beta)|\,d\beta\,d\alpha\to 0, \qquad t \to \tau,
\end{multline*}
uniformly for $\|x\|\le1$ and $\|y\|\le1$.   So, we have (\ref{ABN}), and then (\ref{LimitN})
follows.
\end{proof}

\noindent
{\bf Open question.}
Let us assume that $A$ and $f$ are as in Theorem \ref{Norm1}, so that $f(\cdot A) \in C((0,\infty),L(X))$.    If \eqref{somea} holds and $f(z) = e^{-z}$, then $(f(tA))_{t\ge0}$ is an immediately norm-continuous $C_0$-semigroup, and it is well-known that the spectral mapping theorem then holds in the  form of \eqref{smt}.   We raise the question:  

If $f \in \Bes$ and $f(\cdot A) \in C((0,\infty),L(X))$, does \eqref{smt} hold for $f(tA), \, t>0$?

\subsection{Short-time behaviour}
It is a starting point of semigroup theory that one defines the generator $-A$ of a $C_0$-semigroup as the right-hand derivative (at $t=0$) of the semigroup on its natural domain. Below we show that, when $A$ admits the $\Bes$-calculus, one can recover $A$ as the right-hand $t$-derivative of $f(tA)x$ for any function $f$ from the domain of the generator of the shift semigroup on $\mathcal B$, with $f'(0)\ne0$.   Thus, in the defining relation for $A$ the exponential function can be replaced by a rich supply of functions from $\mathcal B $.

\begin{thm}\label{Gen2}
Let $-A$ be the generator of either a bounded $C_0$-semigroup on a Hilbert space $X$, or a (sectorially) bounded holomorphic $C_0$-semigroup on a Banach space $X$.
Let $f\in D(A_{\mathcal B})$, so 
$f\in \mathcal B$ and $f'\in \mathcal{B}$.    Then, for all $x \in D(A)$, 
\begin{equation} \label{ftd}
\lim_{t\to 0+} t^{-1} (f(tA)x-f(0)x) =f'(0)Ax.
\end{equation}
Conversely, if $f'(0)\ne0$ and 
\[
\lim_{t\to 0+} t^{-1} (f(tA)x-f(0)x)
\]
exists for some $x \in X$, then $x \in D(A)$ and \eqref{ftd} holds.    Thus, if $f'(0)=-1$, then $-A$ coincides with the right-hand derivative of $f(\cdot A)$ at $0$, in the strong operator topology, on its natural domain.
\end{thm}

\begin{proof}
In both cases, $A$ admits the $\Bes$-calculus.  We will use notation corresponding to the duality between $X$ and $X^*$, but in the Hilbert space case, $X^*$ can be identified with $X$ and the duality is the inner product.

Since $f'\in \mathcal{B}$ and $f'(\infty)=0$, the reproducing formula \eqref{bhp} (see also Theorem \ref{rep}) gives
\begin{equation} \label{fp0}
f'(0)=-\frac{2}{\pi}\int_0^\infty \alpha\int_\R \frac{f''(\alpha+i\beta)}{(\alpha-i\beta)^2}\,d\b\,d\alpha.
\end{equation}

Define a family $(F_t)_{t >0}\subset L(X)$ as
\[
F_tx:=t^{-1}(f(tA)x-f(0)x),\qquad t>0,\quad x\in X.
\]
If $x\in D(A)$ and $x^*\in X^*$, then
\begin{align} \label{ftxy}
&\langle F_t x,x^*\rangle  \\
&=-\frac{2}{\pi t}\int_0^\infty \alpha\int_\R
\left( \langle(\alpha-i\beta+tA)^{-2}x,x^*\rangle- \frac{\langle x,x^*\rangle}{(\alpha-i\beta)^2}\right)
f'(\alpha+i\beta)\,d\beta\,d\alpha \nonumber\\
&=\frac{2}{\pi  t}\int_0^\infty \alpha\int_\R
\left( \langle(\alpha-i\beta+tA)^{-1}x,x^*\rangle- \frac{\langle x,x^*\rangle}{\alpha-i\beta} \right)
f''(\alpha+i\beta)\,d\beta\,d\alpha  \nonumber\\
&=-\frac{2}{\pi}\int_0^\infty \alpha\int_\R
\frac{\langle(\alpha-i\beta+tA)^{-1}Ax,x^*\rangle}{\alpha-i\beta} \,
f''(\alpha+i\beta)\,d\beta\,d\alpha. \nonumber
\end{align}
Here the integration by parts is justified, because
\[
\lim_{|\b|\to\infty} \left|(\alpha-i\beta)^{-1}\langle(\alpha-i\beta+tA)^{-1}Ax,x^*\rangle
f'(\alpha+i\beta)\right| = 0,
\]
and the final double integral is absolutely convergent.  Indeed, using \eqref{stes} (or \eqref{cses} in the Hilbert space case),  we  obtain
\begin{align*}
|\langle F_tx,x^*\rangle|
&\le
\frac{2}{\pi}\|x^*\|\int_0^\infty \alpha
\int_\R
\frac{\|(\alpha-i\beta+tA)^{-1}Ax\|}{|\alpha-i\beta|} |f''(\alpha+i\beta)|\,d\beta\,d\alpha\\
&\le \frac{2}{\pi}\|x^*\|\int_0^\infty \alpha \sup_{s\in \R}\, |f''(\alpha+is)|\int_\R
\frac{\|(\alpha-i\beta+tA)^{-1}Ax\|}{|\alpha-i\beta|} \,d\beta\,d\alpha \\
&\le \frac{2}{\pi}\|x^*\|\int_0^\infty \alpha \sup_{s\in \R}\, |f''(\alpha+is)| \, \frac{\pi C_A}{\a} \|Ax\| \,d\a  \\
&\le  2 C_A \|f'\|_\Bes  \|Ax\| \,\|x^*\|,
\end{align*}
where $C_A = M_A$ (or  $K_A$ in the Hilbert space case).   Hence 
\begin{equation}\label{M121}
\|F_tx\|\le 2 C_A \|f'\|_\Bes  \|Ax\|,   \qquad t>0,\quad x\in D(A).
\end{equation}

For $x\in D(A^2)$ and $x^* \in X^*$, \eqref{fp0} and \eqref{ftxy} give
\[
\langle F_tx-f'(0)Ax,x^*\rangle
=-\frac{2}{\pi}\int_0^\infty \alpha \int_\R
K_t(\alpha,\beta;x,x^*)f''(\alpha+i\beta)\,d\beta\,d\alpha,
\]
where
\begin{align*}
K_t(\alpha,\beta;x,x^*)
&= (\alpha-i\beta)^{-1}\langle (\alpha-i\beta+tA)^{-1}Ax,x^*\rangle
-(\alpha-i\beta)^{-2}\langle Ax,x^*\rangle\\
&=-t(\alpha-i\beta)^{-2}\langle(\alpha-i\beta+tA)^{-1}A^2x,x^*\rangle.
\end{align*}
Using both these representations for $K_t$, and either \eqref{stes} or \eqref{cses} for $Ax$ and $A^2x$, we have, for every $\delta>0$, 
\begin{align*}
&|\langle F_tx-f'(0)Ax,x^*\rangle|\\
&\le \frac{2}{\pi}\|x^*\|\int_0^\delta \alpha \int_\R \left(\frac{\|(\alpha-i\beta+tA)^{-1}Ax\|}{ |\alpha-i\beta|} 
+ \frac{\|Ax\|}{|\alpha-i\beta|^{2}}\right) \!|f''(\alpha+i\beta)|\,d\beta\,d\alpha\\
& \null\hskip20pt+\frac{2t}{\pi}\|x^*\|\int_\delta^\infty \alpha \int_\R
 \frac{\|(\alpha-i\beta+tA)^{-1}A^2x\|}{|\alpha-i\beta|^{2}}\, |f''(\alpha+i\beta)|\,d\beta\,d\alpha\\
&\le \frac{2}{\pi}\|x^*\|\int_0^\delta \alpha \sup_{s\in \R}\,|f''(\alpha+is)|
\int_\R \frac{\|(\a - i\b+tA)^{-1}Ax\|}{|\alpha-i\beta|}  \,d\beta \, d\a\\
& \null\hskip20pt+\frac{2}{\pi}\|Ax\|\|x^*\|\int_0^\delta \alpha \sup_{s\in \R}\,|f''(\alpha+is)|\int_\R
|\alpha-i\beta|^{-2}
\,d\beta\,d\alpha\\
& \null\hskip20pt +\frac{2t}{\pi}\|x^*\|\int_\delta^\infty  \sup_{s\in \R}\,|f''(\alpha+is)|
\int_\R \frac{\|(\alpha-i\beta+tA)^{-1}A^2x\|}{|\alpha-i\beta|} \,d\beta\,d\alpha\\
&\le \left(2 C_A + 2\right) \|Ax\|\|x^*\| \int_0^\delta \sup_{s\in \R}\,|f''(\alpha+is)|\,d\alpha\\
& \null\hskip20pt + 2 C_At\|A^2x\|\|x^*\| \int_\delta^\infty \alpha^{-1} \sup_{s\in \R}\,|f''(\alpha+is)|
\,d\alpha.
\end{align*}
Letting first $t\to 0+$ and then $\delta\to 0+$, we obtain that
\begin{equation}\label{F111}
\lim_{t\to 0+}\,\|F_tx-f'(0)Ax\|=0,\qquad x\in D(A^2).
\end{equation}
Then, by (\ref{M121}) and (\ref{F111}), since $D(A^2)$ is a core for  $A$, 
\begin{equation}\label{F1110}
\lim_{t\to 0+}\,\|F_tx-f'(0)Ax\|=0,\qquad x\in D(A).
\end{equation}

For the converse result, we may assume that $f'(0)=1$.  Define the operator $B$ by 
\[
D(B):=\big\{x\in X:\,\lim_{t\to 0+}\,F_tx \, \, \text{exists}\big\},\qquad Bx:=\lim_{t\to 0+}\,F_tx.
\]
Let 
\[
U_n x = n \int_0^{1/n} e^{-tA}x \, dt,  \qquad  n \in \N, \, x\in X,
\]
and note that $U_n(X) \subset D(A)$.   From the result already proved,  $U_n x \in D(B)$ and $AU_nx = BU_nx$.  If $x \in D(B)$, then 
$F_t U_n x = U_n F_t x \to U_n Bx$ as $t\to0+$,
so $B U_n x = U_n Bx$.   As $n \to \infty$, $U_nx \to x$ and $AU_n x \to Bx$.  Since $A$ is closed, $x \in D(A)$ and $Ax = Bx$ as required.
\end{proof}

\subsection{Long-time behaviour}
In this section we bring  the characterisation of exponential stability for operator semigroups
on Hilbert spaces into the framework of the $\Bes$-calculus, and we derive the famous Gearhart--Pr\"uss theorem
as an easy corollary of our more general result.

\begin{thm}\label{ExSt11}
Let $A$ be a linear operator on a Banach space $X$ such that $\sigma(A) \subset \overline{\C}_+$ and
\begin{equation}\label{banach}
\sup_{\alpha >0}\, (\alpha+1)\int_\R
|\langle (\alpha +i\beta+A)^{-2}x,x^{*}\rangle |\, d\beta <\infty 
\end{equation}
for all  $x\in X$ and $x^{*}\in X^{*}$.  
Then, for every $f\in \mathcal{B}_0$, 
\begin{equation} \label{fra}
\lim_{t\to\infty}\,\|f(tA)\|=0.
\end{equation}
Consequently, there exist constants $M$ and $\omega>0$ such that
\begin{equation} \label{exr}
\|e^{-tA}\| \le Me^{-\omega t}, \qquad t\ge0.
\end{equation}
\end{thm}

\begin{proof}
By the closed graph theorem, there exists $C>0$ such that
\[
(\alpha+1)\int_\R |\langle (\alpha +i\beta+A)^{-2}x,x^{*}\rangle |\, d\beta \le C\|x\|\|x^{*}\|,\qquad \alpha >0,
\]
for all  $x\in X$ and $x^{*}\in X^{*}$.
By \eqref{fcdef} for $f(tz)$, we have
\[
\langle f(tA)x,x^{*}\rangle
=-\frac{2}{\pi}\int_0^\infty \alpha \int_\R
\langle (\a-i\b+A)^{-2}x,x^{*}\rangle t
f'(t(\alpha+i\beta))\,d\beta\,d\alpha,
\]
and then
\begin{align*}
|\langle f(tA)x,x^{*}\rangle|&\le
\frac{2C}{\pi}\|x\| \|x^{*}\|\int_0^\infty \frac{\alpha}{\alpha+1}t\sup_{\beta\in \R}\,|f'(t(\alpha+i\beta))|\,d\alpha\\
&=\frac{2C}{\pi}\|x\|\|x^{*}\|\int_0^\infty \frac{s}{s+t}\sup_{\beta\in \R}\,|f'(s+i\beta)|\,ds.
\end{align*}
Since
\[
\lim_{t\to \infty}\,\int_0^\infty \frac{s}{s+t}\sup_{\beta\in \R}\,|f'(s+i\beta)|\,ds=0,
\]
by the monotone convergence theorem, the first statement follows.

For the second statement, take $f(z) = e^{-z}$.  Then $\lim_{t\to\infty} \|e^{-tA}\| = 0$, and \eqref{exr} follows by basic semigroup theory.
\end{proof}

We now give two corollaries of the result above.  The first one is a $\mathcal B$-calculus version of the  Gearhart--Pr\"uss theorem which characterises exponential stability of operator semigroups on Hilbert spaces in resolvent terms, see the final statement of the following.

\begin{cor}
Let $-A$ be the generator of a bounded $C_0$-semigroup on a Hilbert space $X$
such that $(\cdot+A)^{-1}\in H^\infty(\mathbb C_+, L(X))$.
Then for every $f\in \mathcal{B}_0$,
\[
\lim_{t\to\infty}\,\|f(tA)\|=0.
\]
In particular, there exist constants $M$ and $\omega>0$ such that
\begin{equation*}
\|e^{-tA}\| \le Me^{-\omega t}, \qquad t\ge0.
\end{equation*}
\end{cor}

\begin{proof}  
For the first statement, it suffices to establish that  \eqref{banach} holds, and this is a standard argument in proofs of the Gearhart-Pr\"uss theorem (see \cite[Theorem 5.2.1]{ABHN}, for example).  Note first that \eqref{plan} and \eqref{gsf} hold for $\a>0$.  By the resolvent identity,
\[
(\a+is+A)^{-1}x = (1+is+A)^{-1}x + (1-\a)(\a+is+A)^{-1}(1+is+A)^{-1}x
\]
for all $x \in X$.  The assumption on the resolvent of $A$ implies that there are constants $C_1$ and $C_2$ such that
\[
\|(\a+is+A)^{-1}x\| \le C_1\|(1+is+A)^{-1}x\|,
\]
and hence
\[
\int_\R \left\|(\a+i\b+A)^{-1}x\right\|^2 \,d\b \le C_2\|x\|^2, \qquad 0<\a<1, \, x \in X.
\]
Using also the same estimate for $A^*$, it follows that \eqref{gsf} holds for $0<\a<1$ and $n=2$, with $\pi K_A^2 \a^{-1}$ replaced by $C_2$.  Now \eqref{banach} follows.
\end{proof}

\begin{cor} \label{holex}
Let $-A$ be the generator of a bounded holomorphic $C_0$-semigroup on a Banach space, and let $f \in \mathcal B_0$ with $f(0)\neq 0$.   Then \eqref{fra} holds if and only if $0\not \in \sigma(A)$.
\end{cor}

\begin{proof}
Recall that if $-A$ generates a bounded holomorphic semigroup, then $\sigma (A) \subset \mathbb C_+\cup \{0\}$.

Assume that $0\not \in \sigma (A)$.  Then $i\mathbb R \cap \sigma(A)=\emptyset$ and
$(z+A)^{-1}$ is holomorphic near zero. 
From \eqref{R02} we have
\[
\sup_{\alpha >0} \alpha \int_\R \|(\alpha +i\beta+A)^{-2}\|\,d\beta \le \pi M_A^2,
\]
and
\begin{multline*}
\sup_{0 < \alpha \le 1} \int_\R \|(\alpha +i\beta+A)^{-2}\|\,d\beta \\ 
\le \int_{|\beta| \ge 1} \frac{M_A^2}{\b^2} \,d\beta 
+2 \sup_{\alpha \in [0,1], |\b|\le1} \|(\alpha +i\beta +A)^{-2}\|<\infty.
\end{multline*}
So, \eqref{banach} holds and hence  \eqref{fra} holds.

Now assume that $0 \in \sigma(A)$.   By the spectral inclusion given in \eqref{smp_b},
\[
\|f(tA)\|\ge \sup_{z \in \sigma (A)} |f(tz)|\ge |f(0)| > 0
\]
for every $t \ge 0$.   Hence \eqref{fra} does not hold.
\end{proof}

\subsection{Complex inversion  formula}

The complex inversion formula
\begin{equation}\label{inversion}
e^{-tA}x=\lim_{N\to\infty}\frac{1}{2\pi i}\int_{-N}^N e^{t(\sigma+i\b)}(\sigma + i\b +A)^{-1}x \,d\b, \qquad \sigma >0,
\end{equation}
 for (bounded) $C_0$-semigroups $(e^{-tA})_{t \ge 0}$ on Banach spaces is one of the basic tools in the theory of operator semgroups, in particular in the study of their asymptotics.   In general, the limit exists in the classical sense for $x \in D(A)$ and in the Ces\`aro sense for $x\in X$.   Further improvements require additional assumptions either on $X$ or on $A$.
It was discovered in \cite{Yao} that the formula converges for all $x \in X$ if $X$ is a Hilbert space, and this was later  extended to UMD spaces \cite{Dri} and to the more general framework of solution families of integral equations \cite{Haase_I}.   See also \cite[Section 3.12]{ABHN}.

Here we show that \eqref{inversion} has a counterpart for functions $f$ from the domain of the generator $A_\Bes$ of the shift semigroup on $\mathcal B$ (compare Theorem \ref{Gen2}) in terms of the boundary function $f^b$ of $f$.    We formulate the result for $f'(A)$ but the formula for $f'(tA)$ can be obtained by replacing $f(z)$ by $f(tz)$, or $A$ by $tA$.

\begin{prop}\label{F35}
Let $f\in \mathcal{B}$ be such that $f'\in \mathcal{B}$, 
and let  $A$ be an operator admitting the $\mathcal{B}$-calculus on a Banach space $X$.
Then for any $\sigma>0$, $x \in X$ and $x^* \in X^*$,
\begin{equation}\label{a}
\langle f'(A+\sigma)x,x^{*}\rangle=-\frac{1}{2\pi}\int_{\R}\,\langle (\sigma-i\beta+A)^{-2}x,x^{*}\rangle f^b(\beta)\,d\beta.
\end{equation}
Moreover,
\begin{equation}\label{b}
f'(A+\sigma)x=\lim_{N\to\infty}\frac{1}{2\pi}\int_{-N}^N\, (\sigma-i\beta+A)^{-1}x f^b(\b)\,d\beta.
\end{equation}
\end{prop}

\begin{proof}
Using the $\mathcal{B}$-calculus, integration by parts and Proposition \ref{Dd}, we have
\begin{align*}
\langle f'(A+\sigma)x,x^{*}\rangle &=-\frac{2}{\pi}\int_0^\infty \alpha
\int_{\R}\,\langle (\sigma+\alpha-i\beta+A^{-2})x,x^{*}\rangle f''(\alpha+i\beta)\,d\beta\,d\alpha \\
&= \frac{4}{\pi}\int_0^\infty \alpha
\int_{\R}\,\langle (\sigma+\alpha-i\beta+A)^{-3}x,x^{*}\rangle f'(\alpha+i\beta)\,d\beta\,d\alpha\\
&= -\frac{1}{2\pi}\int_{\R}\,\langle (\sigma-i\beta+A)^{-2}x,x^{*}\rangle f^b(\b)\,d\beta.
\end{align*}
Thus \eqref{a} is proved  and the uniform boundedness principle implies that
\[
\sup_{N\in \mathbb N} \Big\|\int_{-N}^{N} (\sigma-i\beta+A)^{-2} f^b(\beta)\,d\beta \Big\|<\infty.
\]
By the resolvent identity,  
\[
\lim_{N \to \infty} \int_{-N}^{N}  (\sigma-i\beta+A)^{-2} (\sigma+A)^{-2}x f^b(\beta)\,d\beta
\]
exists for every $x \in X$.   Since $D(A^2)$ is dense  in $X$, it follows that 
\[
\lim_{N \to \infty} \int_{-N}^{N} \, (\sigma-i\beta+A)^{-2}x f^b(\beta)\,d\beta
\]
exists for all $x \in X$,  and
in view of \eqref{a},  
\[
f'(A+\sigma)x=-\lim_{N\to\infty}\frac{1}{2\pi}\int_{-N}^N\, (\sigma-i\beta+A)^{-2}x \,f^b(\beta)\,d\beta
\]
for all $x \in X$.   By Remark \ref{Bsg}, $-A$ generates a bounded $C_0$-semigroup on $X$.  Thus we obtain \eqref{b} by a further integration by parts, noting that $(\sigma+i\b+A)^{-1}x$ is the Fourier transform of $e^{-\sigma t} e^{-tA}x$ and applying the Riemann-Lebesgue lemma. 
\end{proof}

Setting $f(z)=e^{-tz}$ in Proposition \ref{F35}, we obtain the following statement which is an inversion formula covering both Hilbert space semigroups and bounded holomorphic semigroups.

\begin{cor}\label{41Cor}
Let  $A$ be an operator admitting the $\mathcal{B}$-calculus on a Banach space $X$.
Then, for all $\sigma>0$, $t>0$, $x\in X$ and $x^* \in X^*$,
\begin{align*}
\langle e^{-tA}x,x^{*}\rangle &= \frac{e^{\sigma t}}{2\pi  t}\int_\R e^{it \beta} \langle (\sigma+i\beta+A)^{-2}x,x^{*}\rangle  \,d\beta\\
&= \frac{1}{2\pi i t}\int_{\sigma-i\infty}^{\sigma+i\infty} e^{tz} \langle (z+A)^{-2}x,x^{*}\rangle \,dz.
\end{align*}
Moreover, \eqref{inversion} holds for all $x \in X$.
\end{cor}

\noindent{\bf Concluding remarks.}
For a sectorial operator $A$ and appropriate holomorphic $f$, the families $(f(tA))_{t>0}$ appear frequently in the theory
of functional calculi as an auxiliary tool providing an approximation technique for operator holomorphic functions, the study of interplay between functional calculi and interpolation spaces, and square function estimates (for examples, see \cite{HaaseB}, Section 5.2, Section 6, and Theorem 7.3.1 and Comments for Chapter 7, respectively). The emphasis in our studies of such families is quite different. Given a generator $-A$, we concentrate on characterising the properties of $(f(tA))_{t>0}$ as a function of $t$ and thus revisiting some of the main statements in semigroup theory, i.e., the results for $f(z)=e^{-z}$.   As we have shown above, somewhat surprisingly, those statements have valid  $\mathcal B$-calculus versions, although the arguments become much more involved.

\end{document}